\documentclass[10pt]{amsart}
\usepackage{amsmath}
\usepackage{amssymb}
\usepackage{enumerate}
\usepackage{amsbsy}
\usepackage{amsfonts}
\usepackage{color}

\usepackage[pagewise]{lineno}

\newtheorem{thm}{Theorem}[section]
\newtheorem{lem}[thm]{Lemma}
\newtheorem{prop}[thm]{Proposition}
\newtheorem{cor}[thm]{Corollary}

\newtheorem{rem}[thm]{Remark}

\numberwithin{equation}{section}

\newcommand{\rt}{\mathbb{R}^2}
\newcommand{\ct}{\mathbb{C}^2}

\newcommand{\les}{\lesssim}
\newcommand{\lam}{{\lambda}}
\newcommand{\gam}{{\gamma}}
\newcommand{\Gam}{{\Gamma}}
\newcommand{\vp}{{\varphi}}
\newcommand{\ve}{{\varepsilon}}
\newcommand{\de}{{\delta}}
\newcommand{\al}{{\alpha}}
\newcommand{\ka}{{\kappa}}

\newcommand{\lammin}{{\lambda}_{\rm min}}
\newcommand{\lammed}{{\lambda}_{\rm med}}

\newcommand{\epm}{{e^{\pm it\left<D\right>}}}
\newcommand{\emp}{{e^{\mp it\left<D\right>}}}

\newcommand{\brad}{\left<D\right>}

\newcommand{\psii}{\psi_{\infty}^\ell}

\newcommand{\vps}{\varphi_{\infty}}
\newcommand{\vpsl}{\varphi_{\ka,r}}
\newcommand{\epd}{e^{-it\left< D\right>}}
\newcommand{\emd}{e^{it\left< D\right>}}

\begin{document}

	\title[Scattering for Dirac equations]{Small data scattering of 2d Hartree type Dirac equations}

	\author[Y. Cho]{Yonggeun Cho}
	\address{Department of Mathematics, and Institute of Pure and Applied Mathematics, Jeonbuk National University, Jeonju 54896, Republic of Korea}
	\email{changocho@jbnu.ac.kr}

	\author[K. Lee]{Kiyeon Lee}
	\address{Department of Mathematics, Jeonbuk National University, Jeonju 54896, Republic of Korea}
	\email{leeky@jbnu.ac.kr}

	\author[T. Ozawa]{Tohru Ozawa}
	\address{Department of Applied Physics, Waseda University, 3-4-1, Okubo, Shinjuku-ku, Tokyo, 169-8555, Japan}
	\email{txozawa@waseda.jp}

	
	\begin{abstract}
 In this paper, we study the Cauchy problem of 2d Dirac equation with Hartree type nonlinearity $c(|\cdot|^{-\gamma} * \langle \psi, \beta \psi\rangle)\beta\psi$ with $c\in \mathbb R\setminus\{0\} $, $0 < \gamma <  2$. Our aim is to show the small data global well-posedness and scattering in $H^s$ for $s > \gamma-1$ and $1 < \gamma < 2$. The difficulty stems from the singularity of the low-frequency part $|\xi|^{-(2-\gamma)}\chi_{\{|\xi|\le 1\}}$ of potential. To overcome it we adapt $U^p-V^p$ space argument and bilinear estimates of \cite{yang, tes2d} arising from the null structure. We also provide nonexistence result for scattering in the long-range case $0 < \gamma \le 1$.
	\end{abstract}

		\maketitle

\section{Introduction}
We consider the following Hartree type Dirac equation:
\begin{eqnarray}\label{maineq}
\left\{\begin{array}{l}
(-i\partial_t + \al\cdot D + m\beta) \psi = (V* \left< \psi,\beta\psi\right>)\beta\psi \;\; \mathrm{in}\;\; \mathbb{R}^{1+2},\\
\psi(0) = \psi_0 \in H^s(\rt),
\end{array} \right.
\end{eqnarray}
where $\psi:\mathbb{R}^{1+2} \to \ct$ is the spinor field represented by a column vector $\left(
                                                                                           \begin{array}{c}
                                                                                             \psi_1 \\
                                                                                             \psi_2 \\
                                                                                           \end{array}
                                                                                         \right)$, $D = -i\nabla$, and $\al = (\al^1, \al^2), \beta$ are the Dirac matrices defined by
\begin{align*}
 \alpha^1 = \left(\begin{array}{ll} 0 & \;\;1 \\ 1 & \;\;0 \end{array} \right) \quad \alpha^2 = \left(\begin{array}{ll} 0 & -i \\ i & \;\;0 \end{array} \right),\quad \beta = \left(\begin{array}{ll} 1 & \;\,\,\,0 \\ 0 & -1 \end{array} \right).
\end{align*}
The constant $m \ge 0$ is a physical mass parameter and the symbol $*$ denotes convolution in $\rt$. The potential $V$ is of Coulomb type and defined by $c|x|^{-\gamma}\;(0 < \gamma < 2, c \in \mathbb R \setminus \{0\})$. Throughout the paper $\left< , \right>$ denotes the inner product on $\mathbb C^2$ and $\left< \psi, \phi\right>_{L_x^2} := \int_{\mathbb R^2} \left< \psi, \phi\right>\,dx$. When $\gamma = 1$, \eqref{maineq} can be regarded as a simplified model of Chern-Simons-Dirac system in Coulomb gauge \cite{boucanma}.

It is well-known that every smooth solution to \eqref{maineq} satisfies the mass conservation: $\|\psi(t)\|_{L_x^2} = \|\psi_0\|_{L_x^2}$ for any $t$ within the existence time interval. For instance see \cite{chagla}. In the massless case ($m=0$) the equation \eqref{maineq} has the scaling invariance structure in $\dot{H}^{\frac{\gam-1}2}$. That is, the $\dot{H}^{\frac{\gam-1}2}$-scaled function $\psi_\lambda$ for $\lam > 0$ defined by $\psi_\lam(t,x) = \lam^\frac{3-\gam}{2} \psi (\lam t,\lam x)$ is also the solution to the equation \eqref{maineq}. By this reason, the \eqref{maineq} is said to be mass-critical (supercritical, subcritical) when $\gamma = (>, <)\;1$, respectively.

In this paper we use the representation of solution based on the massive Klein-Gordon equation. Thus we only consider the massive case ($m > 0$) and for simplicity we assume that $m = 1$.

Let us define the energy projection operators $\Pi_\pm(D)$ by
$$
\Pi_{\pm}(D) := \frac12 \left( I \pm \frac1{\brad}[\alpha \cdot D + \beta] \right),
$$
where $\left< D\right> = (1 - \Delta)^\frac12$.
Then we get
\begin{align}
\al\cdot D + \beta = \brad \left(\Pi_{+}(D) - \Pi_{-}(D)\right),
\end{align}
and
\begin{align}\label{proj-commu}
\Pi_{\pm}(D)\Pi_{\pm}(D) = \Pi_{\pm}(D), \;\; \Pi_{\pm}(D)\Pi_{\mp}(D) = 0.
\end{align}
We denote $\Pi_\pm(D)\psi $ by $\psi_\pm$. Then the equation \eqref{maineq} becomes the following system of semi-relativistic Hartree equations:
\begin{align}\label{maineq-decom}
 (-i\partial_t \pm \brad)\psi_\pm = \Pi_{\pm}(D)[(V* \left<\psi,\beta\psi\right>)\beta\psi]
\end{align}
with initial data $\psi_{\pm}(0,\cdot) = \psi_{0, \pm} := \Pi_{\pm}(D)\psi_0$. The free solutions of \eqref{maineq-decom} are $e^{\mp it\brad}\psi_{0, \pm}$, respectively, where
$$
e^{\mp it\langle D\rangle}f(x) = \mathcal F^{-1} e^{\mp it \langle \xi \rangle}\mathcal F f = \frac1{(2\pi)^2}\int_{\mathbb R^2} e^{i(x\cdot \xi \mp t\langle \xi \rangle)}\widehat f(\xi)\,d\xi.
$$
Here $\langle \xi \rangle := (1 + |\xi|^2)^\frac12$ and $\mathcal F, \mathcal F^{-1}$ are Fourier transform, its inverse, respectively.
Then by Duhamel's principle the Cauchy problem \eqref{maineq-decom} is equivalent to solving the integral equations:
\begin{align}\label{inteq0}
\psi_\pm(t) = e^{\mp it\langle D\rangle}\psi_{0, \pm} + i\int_0^t e^{\mp i(t-t')\langle D\rangle}\Pi_{\pm}(D)[(V*\langle \psi, \beta \psi\rangle) \beta \psi](t')\,dt'.
\end{align}
We say that the solution $\psi$ scatters forward (or backward) in a Hilbert space $\mathcal H$ if there exist $\psi^\ell$,  linear solutions to $(-i\partial_t + \alpha \cdot D + \beta)\psi = 0$, such that
\begin{align}\label{sense1}
\|\psi(t) - \psi^\ell(t)\|_{\mathcal H} \to 0\;\;\mbox{as}\;\;t \to +\infty\; (-\infty,\;\;\mbox{respectively}).
\end{align}
In most cases we may take $\mathcal H = H^s$.
Equivalently, $\psi$ is said to scatter forward ( or backward) in $\mathcal H$ if there exist $\psi_{\pm }^\ell := e^{\mp it\brad}\varphi_{\pm } \;(\varphi_{\pm} \in \mathcal H)$ such that
\begin{align}\label{sense2}
\|\psi_\pm(t) - \psi_{\pm}^\ell(t)\|_{\mathcal H} \to 0\;\;\mbox{as}\;\;t \to \pm \infty.
\end{align}

We are concerned with the low-regularity global well-posedness (GWP) and scattering with small initial data for \eqref{maineq} and we equivalently consider  \eqref{inteq0}.
For these questions the Dirac equation with Yukawa potential was studied in \cite{yang, tes2d, tes3d} and \cite{geosha}. Yang \cite{yang} and Tesfahun \cite{tes3d} showed independently scattering results on $H^s(\mathbb R^3)$ for $s>0$ with mass $m \ge 0$.  Tesfahun \cite{tes2d} established scattering in $H^s(\mathbb R^2)$ for $s > 0$ for the massive case. One can find results about power type nonlinearity in \cite{behe, behe2d}. As related topics we refer the readers to \cite{mats1, mats2}, \cite{huhoh, boucanma, holee}, and \cite{choz, coss,  hete, pusa} for the massless Dirac equation with nonlinearity $(|x|^{-\gamma}*|\psi|^{p-1})\psi$, Chern-Simons-Dirac equation in Coulomb gauge, and massive semi-relativistic equation with Coulomb potential, respectively.

Until now not much has been known about the low-regularity issue and the  corresponding scattering theory of Dirac equation \eqref{maineq} with Coulomb type potential. In this paper we suggest a 2d result on the mass-supercritical case as follows.

\begin{thm}\label{mainthm}
Let $1 < \gamma < 2$ and $s > \gam-1$. Then there exists $\ve > 0$ such that for all $\psi_0 \in H^s(\rt)$ satisfying
$$
\|\psi_0\|_{H^s} < \ve,
$$
there exists a unique global solution $\psi \in C(\mathbb{R}; H^s(\rt))$ to \eqref{maineq}. In particular, the global solution $\psi$ scatters both forward and backward in $H^s$.
\end{thm}

For the proof of Theorem \ref{mainthm} we apply the frequency localization to potential $V$ which is two-fold. One is to handle the high-frequency part of $V$ by the null-form and modulation estimates of \cite{tes2d}. The other is to deal with the low-frequency part of $V$ via localized bilinear estimates arising from null-structure described in \cite{behe2}. Since the low-frequency part of Yukawa type potential is bounded, we do not need to handle it. However, the low-frequency part of Coulomb type potential contains significant singularity which is the main obstacle for GWP. To control it we establish the 2d localized bilinear estimates, which are adapted from 3d versions in \cite{heya, yang}. For details see Section \ref{low-esti} below. Our argument can also be applied to the 3d problem in the same way. We leave it to the readers.

In view of the scaling structure there seems to be still a room for GWP and scattering when $\frac{\gamma-1}2 \le s \le \gamma-1$. This issue would be worth pursuing as a future work.


\begin{rem}\label{assumption}
The potential $V$ can be generalized as follows. $\widehat V \in C^\infty(\mathbb R^2 \setminus \{0\})$  and there exist $\gamma_0, \gamma$ such that $1 < \gamma_0 \le 2$, $1\le \gamma < 2$, and
\newcommand{\pgam}{{\gam_0})}
\begin{align*}
\left||\xi|^{k} \partial^k\widehat{V}(\xi) \right| \les   |\xi|^{-(2-\pgam)}\chi_{\{|\xi| \le 1\}} + |\xi|^{-(2-\gam)}\chi_{\{|\xi| > 1 \}}\;(0 \le k \le 4).
\end{align*}
The prototype of $V$ with the property as above is $\mathcal F^{-1}(m_0^2+|\xi|^2)^{-\frac{\gam}2}$ $(m_0 \ge 0, 1 < \gam \le 2)$, which includes the Coulomb type potential $c|x|^{-\gam} \,(m_0 = 0, 1 < \gam < 2)$ and Yukawa type $\mathcal F^{-1}(m_0^2+|\xi|^2)\;(m_0 > 0)$. $V$ does not include two critical potentials: mass-critical case $c|x|^{-1}\;(m_0 = 0, \theta = 1)$, 2d Coulomb potential $-\frac1{2\pi}\ln |x| = \mathcal F^{-1}(|\xi|^{-2})\;(m_0 = 0)$. However, it allows a slight modification of low frequency part of the critical potentials $(\mathcal F^{-1}(|\xi|^{-(1-\varepsilon)}(1+|\xi|)^{-\theta}), \;0 < \varepsilon \ll 1, \theta = \varepsilon$ or $ 1+\varepsilon)$. For this type potential we can get exactly the same conclusion as Theorem \ref{mainthm}. At this point, the low-frequency index $\gamma_0$ is not involved in the regularity because the summation of low-frequency part is carried out on $|\xi| \le 1$ and results in a constant (see Section \ref{lowfrequency} below).
\end{rem}



The mass-supercritical potential can be referred to as a short-range one in the context of scattering theory. In the mass-subcritical (long-range) regime it is usually expected that a number of nontrivial smooth solutions do not behave like a linear solution in the sense of \eqref{sense1} or \eqref{sense2} as time goes on. This turns out to be true for our problem with $0 < \gamma \le 1$. To be more precise, let
$$
\mathcal V(\psi, \phi)(t) = \int  \Big (|\cdot|^{-\gamma}*\left<\psi, \phi\right>\Big)(t, x)  \left<\psi,  \phi\right> (t, x) \;dx.
$$
Then we have the following.

\begin{thm}\label{non-scat}
Let $0 < \gamma < 1$. Suppose that $\psi$ is a smooth global solution to \eqref{maineq} which scatters forward in $L_x^2$ to a smooth linear solution $\psii$.
Suppose that for some $0 < \theta < 1$ and $t_* > 0$ the linear solution $\psii$ satisfies that
	\begin{align}\label{assu-scat}
	 |\mathcal V(\psii, \beta\psii)(t)| \ge \theta \mathcal V(\psii, \psii)(t)
	\end{align}
for any $t > t_*$. Then $\psi, \psii =0$  in $L_x^2$.
Furthermore, if $\gamma = 1$ and $\psi$ is a smooth global solution to \eqref{maineq} which scatters forward in $H^{s, \delta}$ for some $0 < s, \delta \ll 1$ to a smooth linear solution $\psii$ satisfying \eqref{assu-scat}, then $\psi, \psii =0$  in $L_x^2$.
\end{thm}
\noindent Here $H^{s, \delta}$ denotes the space $\{f \in L_x^2 : \|f\|_{H^{s, \delta}} := \|f\|_{H^s} + \|(1+|x|)^\delta f\|_{L_x^2}  < \infty \}$. As for the backward scattering, exactly the same statement holds with the only change $t < t_* < 0$. If $\psi$ scatters in $L_x^2$ or $H^{s, \delta}$ to $\psii$ satisfying \eqref{assu-scat}, then by Lemmas \ref{infty-esti} and \ref{infty-weight} below one can find $0 < \theta' < 1$ and $t_{**} > 0$ such that
\begin{align}\label{trans-psi}
\mathcal V(\psi, \beta \psi )(t) \ge \theta' \mathcal V(\psi, \psi)(t)
\end{align}
for any $t > t_{**}$.

A sufficient condition for \eqref{assu-scat} is that there exists $0 < \theta'' < 1$ such that
$$
\theta''|{\psii}_{, 1}(t, x)| \ge |{\psii}_{, 2}(t, x)|\;\;{\rm or}\;\; \theta''|{\psii}_{, 2}(t, x)| \ge |{\psii}_{, 1}(t, x)|
$$
for any $(t, x) \in \mathbb R^{1 + 2}$ with $t > t_*$, where $\psii = \left(
                                                                        \begin{array}{c}
                                                                          {\psii}_{, 1} \\
                                                                          {\psii}_{, 2} \\
                                                                        \end{array}
                                                                      \right)
$.

In Section \ref{sec-nonsca}, we prove Theorem \ref{non-scat}. Here we explain how the condition \eqref{assu-scat} and $H^{s, \delta}$ assumption come into play. In the proof it is essential to handle the lower bound of functional $H(t)  = {\rm sgn}(c) {\rm Im} \left< \psi(t), \psii(t) \right>_{L_x^2} $ by $\|\psii(0)\|_{L_x^2}$. We will show
$$\left|\frac{d}{dt}H(t)\right| \ge |c| |\mathcal V(\psii, \beta\psii)(t)| + o(t^{-\gamma}).$$ $H^{s, \delta}$ assumption is needed for the remainder term $o(t^{-\gamma})$ when $\gamma = 1$. Due to the matrix $\beta$ the value of $\mathcal V(\psii, \beta\psii)$ may vanish. To avoid this we assume \eqref{assu-scat} and hence obtain
$|\frac{d}{dt}H(t)| \gtrsim t^{-\gamma}\|\psii(0)\|_{L_x^2}^2 + o(t^{-\gamma})$ for sufficiently large $t$.
Therefore, if $0 < \gamma \le 1$ and $\|\psii(0)\|_{L_x^2} > 0$, then the lower bound eventually will lead us to contradiction to the uniform boundedness of $H(t)$.

In the papers \cite{chagla, ozya}, a class of initial data was considered for the nonlinear spinor interaction $\left< \psi, \beta\psi\right>$ to be destructive and hence for the nonlinear term to vanish. This argument can be applied to the equation \eqref{maineq}. In this case the solution to \eqref{maineq} is just linear and is referred to {\it scatter trivially}. As an example we can take the initial data $\psi_0$ satisfying that $A\psi_0 = \overline{\psi_0}$, where $A := \left(\begin{array}{ll} 0 & a_1 \\ a_2 & 0 \end{array} \right)$ for any $a_1,a_2 \in \mathbb{C}$ with $|a_1| = |a_2|=1$. This is the reason why some condition such as (1.8) (consequently \eqref{trans-psi}) is necessary for the nonexsitence of scattering.

In view of Theorem \ref{non-scat} the mass-critical index $\gamma = 1$ is actually scattering-critical. Hence as in many scattering-critical problems we expect a modified scattering theory.



\section{Notation and function space}

\subsection{Notations}\label{nota}
Throughout the paper we use the following notations.

\noindent $(1)$ $\|\cdot\|$ denotes $\|\cdot\|_{L_{t,x}^2}$.\\

\noindent$(2)$ (Mixed-normed spaces) For a Banach space $X$ and an interval $I$, $u \in L_I^q X$ iff $u(t) \in X$ for a.e.$t \in I$ and $\|u\|_{L_I^qX} := \|\|u(t)\|_X\|_{L_I^q} < \infty$. Especially, we denote  $L_I^qL_x^r = L_t^q(I; L_x^r(\rt))$, $L_{I, x}^q = L_I^qL_x^q$, $L_t^qL_x^r = L_{\mathbb R}^qL_x^r$.\\

\noindent$(3)$ (Littlewood-Paley operators) Let $\rho$ be a Littlewood-Paley function such that $\rho \in C^\infty_0(B(0, 2))$ with $\rho(\xi) = 1$  for $|\xi|\le 1$ and define $\rho_{\lam}(\xi):= \rho\left(\frac {\xi}{\lam}\right) - \rho\left(\frac{2\xi}{\lam}\right)$ for $\lam \in 2^\mathbb{Z}$. Then we define the frequency projection $P_\lam$ by $\mathcal{F}(P_\lam f)(\xi) = \rho_{\lam}(\xi)\widehat{f}(\xi)$, and also $P_{\le \mu} := I - \sum_{\lam > \mu}P_{\lam}$. In addition $P_{\mu \le \cdot \le \nu} := \sum_{\mu \le \lambda \le  \nu}P_\lam$ and $P_{\sim \nu} := \sum_{\lambda \sim \nu}P_{\lambda}$. For $ \lam \in 2^\mathbb{Z}$ we denote $\widetilde{\rho_\lam} = \rho_{\frac{\lam}{2}} + \rho_\lam + \rho_{2\lam}$. In particular, $\widetilde{P_\lam}P_\lam = P_\lam\widetilde{P_\lam} = P_\lam$ where $\widetilde{P_\lam} = \mathcal{F}^{-1}\widetilde{\rho_\lam}\mathcal{F}$. Next we define a Fourier localization operators $P_\lam^1$ as follow:
$$
P_\lam^1 f =  \left\{ \begin{array}{ll} 0  &\;\;\;\mbox{if}\;\; \lam < 1 ,\\ P_{\le 1}f &\;\;\;\mbox{if}\;\; \lam = 1 ,\\ P_\lam f & \;\;\;\mbox{if}\;\; \lam > 1.   \end{array} \right .
$$
Especially, we denote $P_\lam^1 f $ by $f_{\lam}$ for any measurable function $f$.\\

\noindent$(4)$ As usual different positive constants depending only on $\gamma_0, \gamma$ are denoted by the same letter $C$, if not specified. $A \lesssim B$ and $A \gtrsim B$ mean that $A \le CB$ and
$A \ge C^{-1}B$, respectively for some $C>0$. $A \sim B$ means that $A \lesssim B$ and $A \gtrsim B$.

\subsection{$U^p - V^p$ and adapted function space}
We explain concisely $U^p-V^p$ spaces. For more details we refer the readers to \cite{ haheko,haheko2, kota, kotavi}.
Let $1 \le p < \infty$ and $\mathcal{I}$ be the collection of finite partitions $\{ t_0 , \cdots, t_N\}$ satisfying $-\infty < t_0 < \cdots < t_N \le \infty$. If $t_N =\infty$, by convention, $u(t_N) := 0$ for any $u : \mathbb R \to L_x^2(\mathbb R^2)$.
Let us define a $U^p$-atom by a step function $a : \mathbb{R} \to L_x^2$ of the form
$$
a(t) = \sum_{k=1}^N \chi_{[t_{k-1},t_k)}\phi_k(t) \;\;\mbox{with}\;\; \sum_{k=1}^N \|\phi_k\|_{L_x^2}^p=1.
$$
Then the $U^p$ space is defined by
$$
U^p = \left\{ u = \sum_{j=1}^\infty \lam_j a_j : \mbox{$a_j$ are $U^p$-atoms and $\{\lam_j\} \in \ell^1$ },  \|u\|_{U^p}< \infty \right\},
$$
where the $U^p$-norm is defined by
$$
\|u\|_{U^p}:= \inf_{\mbox{representation of $u$}} \;\;\sum_{j=1}^\infty|\lam_j|.
$$
We next define $V^p$ as the space of all right-continuous functions $v : \mathbb{R} \to L_x^2$ satisfying that $\underset{t \to -\infty}{\lim}v(t) =0$ and the norm
$$
\|v\|_{V^p} := \sup_{\{t_k\} \in \mathcal{I}} \left( \sum_{k=1}^N \|v(t_k)  - v(t_{k-1})\|_{L_x^2}^p \right)^{\frac1p}
$$
is finite.

We introduce some properties of $U^p$ and $V^p$.
\begin{lem}[\cite{haheko}]\label{embedd} Let $1 \le p < q <\infty$. Then the following holds.
	\item[(i)] $U^p$ and $V^p$ are Banach spaces.
	\item[(ii)] The embeddings $U^p \hookrightarrow V^p \hookrightarrow U^q \hookrightarrow L^{\infty}(\mathbb{R};L_x^2) $ are continuous.
\end{lem}
These spaces have the useful duality property.
\begin{lem}[Corollary of \cite{kotavi}]\label{duality} Let $u \in U^p$ be absolutely continuous with $1 < p < \infty$. Then
	$$
	\|u\|_{U^p} = \sup \left\{ \left| \int \left<u', v \right>_{L_x^2} dt \right| : v \in C_0^{\infty}, \;\; \|v\|_{V^{p'}}=1 \right\}.
	$$
\end{lem}

Now let us define the adapted function spaces $U_{\pm}^p,\; V_{\pm}^p$ as follows:
$$
\|u\|_{U_{\pm}^p} : =\|\epm u\|_{U^p} \;\;\mbox{and}\;\;  \|u\|_{V_{\pm}^p} : =\|\epm u\|_{V^p}.
$$

\begin{lem}[Logarithm interpolation, Lemma 4.12 of \cite{kotavi}]\label{log} Let $p > 2$ and $v \in V_{\pm}^2$. Then there exist a constant $\kappa(p) > 0$ and functions $u \in U_\pm^2$ and $w \in U_{\pm}^p$ such that
	$$
	v = u +w
	$$
	and for all $M\ge1$
	$$
	\frac{\kappa}{M} \|u\|_{U_\pm^2} + e^M\|w\|_{U_{\pm}^p} \les \|v\|_{V_{\pm}^2}.
	$$
\end{lem}

\begin{prop}[Transfer principle, Proposition 2.19 of \cite{haheko}]\label{transfer} Let
	$$
	T : L_x^2 \times L_x^2 \times \cdots \times L_x^2 \to L_{loc}^1
	$$
	be a multilinear operator. If
	$$
	\left\|T\left( e^{\pm_1 it\brad}f_1 , e^{\pm_2 it\brad}f_2 , \cdots ,e^{\pm_k it\brad}f_k \right)\right\|_{L_t^qL_x^r} \les \prod_{j=1}^{k}\|f_j\|_{L_x^2}
	$$
	for some $1 \le q,r \le \infty$, then we have
	$$
	\|T(u_1, u_2,\cdots, u_k)\|_{L_t^q L_x^r} \les \prod_{j=1}^k \|u_j\|_{U_{\pm_j}^q},
	$$
	where $u_j$ are arbitrary functions in $U_{\pm_j}^q$.
\end{prop}

\newcommand{\etahat}{\widehat{\eta}}
\newcommand{\sighat}{\widehat{\sigma}}
\newcommand{\brae}{\left<\eta \right>}
\newcommand{\bras}{\left<\sigma \right>}

\section{Null-form and bilinear estimates}

 The null structure of \eqref{maineq} was revealed concretely in \cite{tes2d} and can be described as follows: 
\begin{align*}
\left<  \Pi_{+}(D) \psi_1 , \beta \Pi_{\pm}(D)\psi_2 \right>
= \sum_{j=1}^3 Q_j^\pm(\psi_1,\psi_2)    + \sum_{j=1}^4 B_j^\pm (\psi_1, \psi_2),
\end{align*}
where the null forms $Q_j^\pm$ and bilinear forms $B_j^\pm$ are given by
\begin{align*}
 Q_j^\pm(\psi_1,\psi_2)&=  \mathcal{F}^{-1} \iint_{\eta-\sigma = \xi} \left< q_j^\pm (\eta,\sigma)   \widehat{\psi_1}(\eta) ,  \beta \widehat{\psi_2}(\sigma)  \right> \,d\eta d\sigma    ,\\
 B_j^\pm (\psi_1, \psi_2)&= \mathcal{F}^{-1} \iint_{\eta-\sigma = \xi} \left<  b_j^\pm (\eta,\sigma)  \widehat{\psi_1}(\eta) ,  \beta \widehat{\psi_2}(\sigma)  \right> \,d\eta d\sigma
\end{align*}
 for $j= 1,2,3$, and
$$
 B_4^\pm (\psi_1, \psi_2)= \pm  \iint_{\eta-\sigma = \xi}  \left<\eta\right>^{-1} \left<  \widehat{\psi_1^+}(\eta) ,  \widehat{\psi_2^\pm}(\sigma) \right>\,d\eta d\sigma,
$$
where
\begin{align*}
q_1^{\pm}(\eta, \sigma) &:=  (|\etahat||\sighat| \pm \etahat\cdot \sighat)I,   \qquad\qquad   b_1^{\pm}(\eta, \sigma) :=  (1- |\etahat||\sighat|)I,\\
q_2^{\pm}(\eta, \sigma) &:=  -(\etahat|\sighat| \pm \sighat|\etahat|)\cdot\al, \qquad\quad   b_2^{\pm}(\eta, \sigma) :=  - \frac{[\eta(\bras - |\sigma|) \pm \sigma(\brae - |\eta|)]\cdot\al}{\brae\bras},\\
q_3^{\pm}(\eta, \sigma) &:=  \pm(\etahat_1\sighat_2 - \etahat_2 \sighat_1)\al^1\al^2,   \;\;\;\;\;\;   b_3^{\pm}(\eta, \sigma) :=  \pm \frac{[(\eta-\sigma)\cdot\al]\beta - (\brae \pm \bras)\beta + I}{\brae \bras}.\\
\end{align*}
Here we used the normalization $\widehat{\xi} = \frac{\xi}{\left<\xi \right>}, \widehat{\xi}_j = \frac{\xi_j}{\left<\xi \right>}\; (j =  1, 2 )$.

\medskip

We now provide localized null-form estimates and bilinear estimates.


\subsection{High-frequency estimates}

\begin{lem}[Null-form estimates, Corollary 1 of \cite{tes2d}]\label{null-esti}
Let $\mu,\lam_1,\lam_2 \in 2^{\mathbb{Z}}$ satisfy $\mu, \lam_1, \lam_2 \ge 1$ and $\lammin$ be minimum value of $\lam_1$ and $\lam_2$.  Then the following holds:
\item[(i)] 
\begin{align*}
\left\|P_\mu Q_j^+ (\psi_{\lam_1} , \phi_{\lam_2})\right\| \les \left\{\begin{array}{ll}
\lammin^{\frac12} \|\psi_{\lam_1}\|_{U_+^2}\|\phi_{\lam_2}\|_{U_+^2} & \mbox{for}\;\; \lam_1 \nsim \lam_2, \\ \mu\lam_{1}^{-\frac12} \|\psi_{\lam_1}\|_{U_+^2}\|\phi_{\lam_2}\|_{U_+^2} & \mbox{for}\;\; \lam_1 \sim \lam_2
\end{array}\right.
\end{align*}
for any $\psi_{\lam_1}, \phi_{\lam_2} \in U_+^2$.
\item[(ii)]  \begin{align*}
\left\|P_\mu Q_j^- (\psi_{\lam_1} , \phi_{\lam_2})\right\| \les \left\{\begin{array}{ll}
\lammin^{\frac12} \|\psi_{\lam_1}\|_{U_+^2}\|\phi_{\lam_2}\|_{U_-^2} & \mbox{for}\;\; \lam_1 \nsim \lam_2, \\ \mu^{\frac12} \|\psi_{\lam_1}\|_{U_+^2}\|\phi_{\lam_2}\|_{U_-^2} & \mbox{for}\;\; \lam_1 \sim \lam_2
\end{array}\right.
\end{align*}
for any $\psi_{\lam_1} \in U_+^2, \phi_{\lam_2} \in U_-^2$.
\end{lem}


\begin{lem}[Lemma 10 of \cite{tes2d}]\label{bilinear} Let $\psi_{\lam_1},\phi_{\lam_2} \in V_{\pm}^2$ and  $Q,B$ be any one of $Q_j^\pm$'s and $B_j^\pm$'s, respectively. Then we have
\begin{align*}
\|P_\mu Q(\psi_{\lam_1} , \phi_{\lam_2})\| &\les \lam_1^{1-\frac{2}{r}} \lam_2^{\frac2r} \|\psi_{\lam_1}\|_{U_\pm^{\frac{2r}{r-2}}} \|\phi_{\lam_1}\|_{U_\pm^{r}} \quad\mbox{for all}\;\; 2<r<\infty, \\
\|P_\mu B(\psi_{\lam_1} , \phi_{\lam_2})\| &\les \frac{\lam_1^{1-s} \lam_2^{s}}{\lammin} \|\psi_{\lam_1}\|_{V_\pm^2} \|\phi_{\lam_1}\|_{V_\pm^2} \qquad\mbox{for all}\;\; 0<s<1.
\end{align*}
In particular
\begin{align}\label{endpoint}
\|P_\mu B(\psi_{\lam_1} , \phi_{\lam_2})\| &\les \frac{\lam_1^{1-s} \lam_2^{s}}{\lammin} \|\psi_{\lam_1}\|_{U_\pm^2} \|\phi_{\lam_1}\|_{U_\pm^2}  \qquad\mbox{for} \;s=0,1.
\end{align}
\end{lem}

\subsection{Low-frequency estimates}\label{low-esti}
We introduce bilinear estimates for free waves, which are crucial for the low-frequency estimates of potential.  In \cite{yang, heya} one can find a 3d versions Corollary \ref{bilinear2} below. We adopt a similar argument to extend the 3d estimates to the 2d ones.

We first make a decomposition on the unit circle as in \cite{ste,yang}. For a fixed collection $ \Omega_\ka :=\{ \xi_{\ka}^\nu\}_\nu$ of unit vectors satisfying the conditions:
\begin{enumerate}
	\item[$(1)$] $|\xi_{\ka}^\nu - \xi_{\ka}^{\nu'}| \ge \ka^{-1}$ if $\nu \neq \nu'$,
	\item[$(2)$] For any unit vector $\xi$, there exists a $\xi_{\ka}^\nu$ such that $|\xi - \xi_{\ka}^\nu| < \ka$,
\end{enumerate}
we set $k_\ka^\nu := \sigma_\ka^\nu \left( \sum_{\ka} \sigma_{\ka}^\nu\right)^{-1}$ and $K_\ka^\nu := \mathcal{F}^{-1} k_{\ka}^\nu \mathcal{F}$ and $\sigma_\ka^\nu = \rho\left( \ka (\frac{\xi}{|\xi|} - \xi_{\ka}^\nu)\right)$, where $\rho$ is the cut-off function as stated in notation (3).  Then we get the following lemma.

\begin{lem}[Lemma 3.3 of \cite{behe2}] For any $\xi_{\ka}^{\nu},\xi_{\ka}^{\nu'} \in \Omega_\ka$ with $|\pm_1\xi_{\ka}^{\nu} - \pm_2 \xi_{\ka}^{\nu'}|\le \ka^{-1}$, we obtain
	$$
	\left| \left< \Pi_{\pm_1}(\lam_1 \xi_{\ka}^{\nu})v , \beta \Pi_{\pm_2}(\lam_2 \xi_{\ka}^{\nu'})w \right> \right| \les \ka^{-1} |v||w|
	$$
	for all $v,w \in \mathbb{C}^2$.
\end{lem}


Next we localize the Klein-Gordon Strichartz estimates. For this purpose let us define a frequency localization operator $\Gam$. Let $\gam \in C_0^\infty(-\frac23,\frac23)$ with $\gam(s) = 1$ if $|s| \le \frac13$ and $B_{\lam'} = \lam'\mathbb{Z}^2$ for $\lam' \in 2^{\mathbb{Z}}$. We define a cut-off function $\gam_{\lam',n}$ such as $\gam_{\lam',n}(\xi) = \gam \left(\frac{\xi_1 - n_1}{\lam'}\right) \gam \left(\frac{\xi_2 - n_2}{\lam'}\right)$, where $\xi = (\xi_1, \xi_2)$ and $n = (n_1,n_2)$ for $\xi \in \rt$ and $n \in B_{\lam'}$ and we prescribe the frequency localization operator $\Gam_{\lam',n} = \mathcal{F}^{-1}\gam_{\lam',n}\mathcal{F}$. Then we get the following.

\begin{lem} Let  $(q,r)$ satisfy that $\frac1q + \frac1r = \frac12$. Then
	\begin{align}\label{kg-stri2}
	\|\Gam_{\lam',n}\epm f_\lam\|_{L_t^q L_x^r} \les (\lam')^{\frac2q}\|f_\lam\|_{L_x^2}
	\end{align}
	for all $\lam,\lam' \in 2^\mathbb{Z}(\lam' \le \lam)$.	
\end{lem}
This lemma follows immediately from the standard Strichartz estimate:
\begin{lem}[Klein-Gordon Strichartz estimates \cite{chozxi}]\label{KG-stri} Let $(q,r)$ satisfy that $\frac1q + \frac1r = \frac12$. Then
\begin{align}\label{KG-stri}
\|\epm f_\lam\|_{L_t^q L_x^r} \les \lam^{\frac2q}\|f\|_{L_x^2}
\end{align}
for all $\lam \in 2^\mathbb{Z}$.		
\end{lem}
Now we extend the localized Klein-Gordon Strichartz estimates globally.
\begin{lem}
	Let $(q,r)$ satisfy that $\frac1q + \frac1r = \frac12$ and $2<q< \infty$. Then
	\begin{align}\label{localizing-stri}
	\left( \sum_{\nu\in \Omega_\ka}\sum_{n \in B_{\lam'}} \|K_l^\nu \Gam_{\lam',n}  \psi_\lam\|_{L_t^qL_x^r}^2 \right)^{\frac12} \les ( \lam')^\frac2q \|\psi_\lam\|_{V_{\pm}^2}.
	\end{align}
\end{lem}
\begin{proof}
By Proposition \ref{transfer} and \eqref{kg-stri2}, we get
\begin{align*}
\|\gam_{\lam',n}\epm \psi_\lam\|_{L_t^q L_x^r} \les ( \lam')^\frac2q \|\psi_\lam\|_{U_{\pm}^q} \les ( \lam')^\frac2q \|\psi_\lam\|_{V_{\pm}^2}.
\end{align*}
From the definition of operators $P_\lam^1, \,\Gam_{\lam',n} $ we see that
\begin{align*}
\| K_{\ka}^\nu \Gam_{\lam',n} \psi_{\lam}  \|_{L_t^qL_x^r} \les  ( \lam')^\frac2q \|K_{\ka}^\nu \Gam_{\lam',n} \psi_{\lam} \|_{V_{\pm}^2}.
\end{align*}
Then, by the orthogonality with respect to $n,\nu$, we obtain
\begin{align*}
\left( \sum_{\nu\in \Omega_\ka}\sum_{n \in B_{\lam'}} \|K_l^\nu \Gam_{\lam',n}  \psi_\lam\|_{L_t^qL_x^r}^2 \right)^{\frac12} &\les (\lam')^{\frac2q} \left( \sum_{\nu\in \Omega_\ka}\sum_{n \in B_{\lam'}} \|K_l^\nu \Gam_{\lam',n}  \psi_\lam\|_{V_{\pm}^2}^2 \right)^{\frac12}\\
&\les (\lam')^{\frac2q} \|\psi_\lam\|_{V_{\pm}^2}.
\end{align*}
\end{proof}

\begin{prop}\label{non-resonance}
Let  $\mu,\lam_1,\lam_2 \in 2^{\mathbb{Z}}$ satisfy $\mu \ll \lam_1 \sim \lam_2$ and $\lam_1,\lam_2 \ge 1$. Suppose $\pm_1 = \pm_2$.  Assume that $\psi_{j,\lam_j} = \Pi_{\pm_j}(D)\psi_{j,\lam_j} \in V_{\pm_j}^2$ for $j = 1,2$. Then
\begin{align*}
\left\|P_\mu \left< \psi_{1,\lam_1}, \beta\psi_{2,\lam_2} \right> \right\| \les \mu\|\psi_{1,\lam_1}\|_{V_{\pm_1}^2} \|\psi_{2,\lam_2}\|_{V_{\pm_2}^2}.
\end{align*}
\end{prop}

\begin{proof}
We decompose $P_\mu \left< \psi_{1,\lam_1}, \beta\psi_{2,\lam_2} \right>$ with respect to $n,n' \in B_{\mu}$, $\nu,\nu' \in \Omega_\ka$ with $\ka \sim \mu^{-1}\lam_1$. Only those $n' \in B_{\mu}$  with $|n-n'|\le \mu$ are chosen for each $n \in B_\mu$ and we carry out the summation under a support condition $\left|\pm_1 \xi_\ka^\nu - \pm_2 \xi_\ka^{\nu'} \right| = |\xi_1 - \xi_2| \sim \ka^{-1}$. Then we estimate
\begin{align*}
\|P_\mu \left< \psi_{1,\lam_1}, \beta\psi_{2,\lam_2} \right> \| &\les  \sum_{\xi_\ka^\nu , \xi_\ka^{\nu'} \in \Omega_\ka} \sum_{n,n' \in B_\mu} \left\|P_\mu \left< K_{\ka}^\nu \Gam_{\mu,n}\psi_{1,\lam_1}, \beta K_{\ka}^{\nu'} \Gam_{\mu,n} \psi_{2,\lam_2} \right> \right\|\\
&\les  \sum_{\xi_\ka^\nu , \xi_\ka^{\nu'} \in \Omega_\ka} \sum_{n,n' \in B_\mu} \ka^{-1} \left\|K_{\ka}^\nu \Gam_{\mu,n}\psi_{1,\lam_1}\right\|_{L_{t,x}^4}  \left\| K_{\ka}^{\nu'} \Gam_{\mu,n} \psi_{2,\lam_2} \right\|_{L_{t,x}^4}.
\end{align*}
Therefore, using Cauchy-Schwarz inequality with $n,n'$ and $\nu, \nu'$ and \eqref{localizing-stri}, we obtain
$$
\|P_\mu \left< \psi_{1,\lam_1}, \beta\psi_{2,\lam_2} \right> \| \les \mu (\lam_1)^{-1} \mu^{\frac12}\mu^{\frac12}\|\psi_{1,\lam_1}\|_{V_{\pm_1}^2}\|\psi_{2,\lam_2}\|_{V_{\pm_2}^2} \les \mu \|\psi_{1,\lam_1}\|_{V_{\pm_1}^2}\|\psi_{2,\lam_2}\|_{V_{\pm_2}^2}.
$$
\end{proof}

\begin{lem}\label{l2-lem} Let  $\mu,\lam_1,\lam_2 \in 2^{\mathbb{Z}}$ satisfy $\mu \ll \lam_1 \sim \lam_2$ and $\lam_1,\lam_2 \ge 1$.  Assume that $f_{\lam_1},g_{\lam_2} \in L_x^2$. Then
   \begin{align}\label{l2-esti}
	\left\|P_\mu \left( e^{\pm_1it\brad}f_{\lam_1}\cdot e^{\pm_2it\brad}g_{\lam_2} \right) \right\|_{L_x^2} \les \mu^{\frac12}\|f_{\lam_1}\|_{L_x^2} \|g_{\lam_2}\|_{L_x^2}
   \end{align}
 with $\pm_1 = \pm_2$. 
\end{lem}
\begin{proof}
We assume that $\pm_1 = \pm_2 = +$. Then
	\begin{align*}
	\left\|P_{\mu} \left( e^{it\brad} f_{\lam_1}\cdot e^{it\brad} g_{\lam_2} \right) \right\| \les \sum_{\begin{subarray} ||n_i| \sim \lam_i \\ |n_1 + n_2|  \sim \mu\end{subarray}} \|I_{n_1,n_2}\|,
	\end{align*}
	where
	$$
	I_{n_1,n_2}(t,x) = \left( \Gam_{\mu,n_1}e^{it\brad} f_{\lam_1}\cdot \Gam_{\mu,n_2}e^{it\brad} g_{\lam_2} \right).
	$$
	Then $I_{n_1,n_2}$ is written as
	\begin{align*}
	I_{n_1,n_2}(t,x) = I_{n_1,n_2}^1 (t,x) +  I_{n_1,n_2}^2 (t,x),
	\end{align*}
	where
	\begin{align*}
	I_{n_1,n_2}^i = \int_{A_i} \int e^{i(t,x)\cdot(\left< \eta \right>+\left< \xi \right>,\, \xi+\eta)}  \left(\rho_{\lam_1}(\xi) \gam_{\mu,n_1}(\xi)\widehat{f}(\xi)  \cdot \rho_{\lam_2}(\eta) \gam_{\mu,n_2}(\eta)\widehat{g}(\xi)  \right) d\eta d\xi \;\;\;\; (i=1,2).
 	\end{align*}
 	Here we denote $A_i$ by
 	$$
 	A_i \subset \{  \xi = (\xi_1 , \xi_2)\in \mathbb{R}^2 : |\xi| \sim |\xi_i|  \}\;\mbox{for}\;i=1,2 \;\;\;\mbox{and}\;\;\; A_1\cup A_2 = \mathbb{R}^2.
 	$$
 	We make the change of variables $(\xi_i, \eta) \rightarrow (\left< \eta \right>+\left< \xi \right>,\xi+\eta)= \zeta = (\zeta_1,\zeta_2,\zeta_3)$ with
 	\begin{align}\label{yacobi}
 	d\xi_i d\eta = \left|\frac{\partial(\zeta_1,\zeta_2,\zeta_3)}{\partial(\eta_1,\eta_2,\xi_i)} \right|^{-1}d\zeta.
 	\end{align}
 	Since $|n_1 + n_2 | \sim \mu$, $\left|\frac{\partial(\zeta_1,\zeta_2,\zeta_3)}{\partial(\eta_1,\eta_2,\xi_i)} \right| \sim 1$. Now we fix $i=1$. The other cases can be estimated similarly. We may estimate by Minkowski's inequality
 	\begin{align*}
 	\|I_{n_1,n_2}^1\| \les \int_{|\xi_2| \les \mu} \left\| \iint e^{i(t,x)\cdot(\left< \eta \right>+\left< \xi \right>,\xi+\eta)}  \left(  \rho_{\lam_1}(\xi) \gam_{\mu,n_1}(\xi)\widehat{f}(\xi)  \cdot \rho_{\lam_2}(\eta) \gam_{\mu,n_2}(\eta)\widehat{g}(\xi)  \right)   d\eta d\xi_1\right\| d\xi_2.
 	\end{align*}
 	Using Plancherel's theorem with respect to  $(t,x)$, the change of variables $\zeta \to (\xi_1,\eta)$, and \eqref{yacobi}, we get
 	$$
 	\|I_{n_1,n_2}^1\| \les \int_{|\xi_2| \les \mu} \left\|  \left( \rho_{\lam_1}(\xi) \gam_{\mu,n_1}(\xi)\widehat{f}(\xi)  \cdot \rho_{\lam_2}(\eta) \gam_{\mu,n_2}(\eta)\widehat{g}(\xi)  \right)  \right\|_{L_{\eta,\xi_1}^2} d\xi_2.
 	$$
 	And also we have
 	\begin{align*}
 	\|I_{n_1,n_2}^2\| \les \mu^{\frac12} \left\|\left( \rho_{\lam_1}(\xi) \gam_{\mu,n_1}(\xi)\widehat{f}(\xi)  \cdot \rho_{\lam_2}(\eta) \gam_{\mu,n_2}(\eta)\widehat{g}(\xi)  \right)  \right \|_{L_{\eta,\xi}^2} \les \mu^{\frac12} \|\Gam_{\mu,n_1} f_{\lam_1}\|_{L_x^2} \|\Gam_{\mu,n_2} g_{\lam_2}\|_{L_x^2}.
 	\end{align*}
 	These estimates yield that
 	$$
 	\left\|P_{\mu} \left( e^{it\brad} f_{\lam_1}\cdot e^{it\brad} g_{\lam_2} \right) \right\| \les \sum_{\begin{subarray} ||n_i| \sim \lam_i \\ |n_1 + n_2|  \sim \mu\end{subarray}}  \mu^{\frac12} \|\Gam_{\mu,n_1} f_{\lam_1}\|_{L_x^2} \|\Gam_{\mu,n_2} g_{\lam_2}\|_{L_x^2} \les \mu^{\frac12} \|f_{\lam_1}\|_{L_x^2} \|g_{\lam_2}\|_{L_x^2}.
 	$$
 	This completes the proof of Lemma \ref{l2-lem}.
 \end{proof}

\begin{prop}\label{resonance} Let  $\mu,\lam_1,\lam_2 \in 2^{\mathbb{Z}}$ satisfy $\mu \ll \lam_1 \sim \lam_2$ and $\lam_1,\lam_2 \ge 1$. Suppose $\pm_1 \neq \pm_2$.  Assume that $\psi_{j,\lam_j} = \Pi_{\pm_j}(D)\psi_{j,\lam_j} \in V_{\pm_j}^2$ for $j = 1,2$. Then
	\begin{align}\label{eq-bilinear2}
	\left\|P_\mu \left< \psi_{1,\lam_1}, \psi_{2,\lam_2} \right> \right\| \les \mu^{\frac12}\|\psi_{1,\lam_1}\|_{V_{\pm_1}^2} \|\psi_{2,\lam_2}\|_{V_{\pm_2}^2}.
	\end{align}
\end{prop}

\begin{proof}
	Since $\epm \psi = \overline{\emp \overline{\psi}}$, $\|\psi\|_{V_{\pm}^2} = \|\overline{\psi}\|_{V_{\mp}^2}$. Hence it suffices to prove that
	\begin{align}\label{realproduct}
	\left\|P_\mu \left( \psi_{1,\lam_1} \cdot \psi_{2,\lam_2} \right) \right\| \les \mu^{\frac12}\|\psi_{1,\lam_1}\|_{V_{\pm_1}^2} \|\psi_{2,\lam_2}\|_{V_{\pm_2}^2}.
	\end{align}
	for $\pm_1 = \pm_2 =+$.	We first show that  for any $\mu \les \lam_1$
	\begin{align}\label{claim-bilinear}
	\|P_{\le\mu} \left( P_{\sim \lam_1}\psi \cdot P_{\sim \lam_1}\psi \right)\| \les \mu^{\frac12} \|\psi\|_{V_{+}^2}^2
	\end{align}
	for real-valued function $\psi \in V_+^2$. To this end let $\omega_\mu = \left[\mathcal{F}^{-1}\left(\rho\left(\frac{\cdot}{2\mu}\right)\right)\right]^2$. Then
	$$
	\|P_{\le\mu}\left( P_{\sim \lam_1}\psi \cdot P_{\sim \lam_1}\psi\right)\| \les \|\omega_\mu*\left( P_{\sim \lam_1}\psi \cdot P_{\sim \lam_1}\psi\right)\| .
	$$
	Hence it suffices to show that
	\begin{align*}
	\|\omega_\mu * \left( P_{\sim \lam_1}\psi \cdot P_{\sim \lam_1}\psi\right)\| \les \mu^{\frac12}\|\psi\|_{V_{+}^2}^2.
	\end{align*}
	 We define
	$$
	n(f) := \| \omega_\mu* \left( P_{\sim \lam_1}f \cdot P_{\sim \lam_1}f\right) \|_{L_x^2}^2.
	$$
	Then we get $n(f+g) \le n(f) + n(g)$. Using \eqref{l2-esti} we see that
	\begin{align}\label{n-linear}
	\|n(e^{-it\left<D\right>}f)\|_{L_t^4} \les \left\|\omega_\mu* \left( e^{-it\left<D\right>}f \cdot e^{-it\left<D\right>}f \right) \right\|_{L_{t,x}^2}^{\frac12} \les \mu^{\frac14} \|f\|_{L_x^2}
	\end{align}
	for all $f\in L_x^2$. Let $\psi \in U_{+}^4$ be with atomic decomposition
	$$
	e^{it\brad}\psi = \sum c_ja_j,\;\; \sum|c_j| \le 2 \|\psi\|_{U_{+}^4},
	$$
	where $a_j$ are $U^4$-atoms. Then we get
	\begin{align}
	\|n(\psi)\|_{L_t^4} \le \sum_{j} |c_j|\|n(e^{-it\brad}a_j)\|_{L_t^4} \les \mu^{\frac14}\|\psi\|_{U_{+}^4},
	\end{align}
	provided that for any $U^4$-atom $a$ the estimate
	$$
	\|n(e^{-it\brad}a)\|_{L_t^4} \les \mu^{\frac14}
	$$
	holds true. Indeed, let $a(t) = \sum_{k}\chi_{I_k}(t) \phi_k $, for some time partition $I_k \subset \mathbb{R}$ and $\phi_k \in L_x^2$ satisfying $\sum_{k}\|\phi_k\|^4\le 1$. Then by \eqref{n-linear} we obtain
	\begin{align*}
	\|n(e^{-it\brad} a)\|_{L_t^4} &\le \left\|\sum_{k}\chi_{I_k}(t) n(e^{-it\left<D\right>}\phi_k)\right\|_{L_t^4} \le \left(\sum_{k}\left\| n(e^{-it\left<D\right>}\phi_k)\right\|_{L_t^4}^4\right)^{\frac14}\\
	&\les \mu^{\frac14} \left(\sum_{\mu}\left\| \phi_\mu\right\|_{L_x^2}^4\right)^{\frac14} \les \mu^{\frac14}
	\end{align*}
	and hence
	$$
	\|\omega_\mu * \left( P_{\sim \lam_1}\psi \cdot P_{\sim \lam_1}\psi \right)\| = \|n(\psi(t))\|_{L_t^4}^2 \les \mu^{\frac12}\|\psi\|_{U_{+}^4}^2 \les \mu^{\frac12}\|\psi\|_{V_{+}^2}^2.
	$$
	This leads us to \eqref{claim-bilinear}.
	
	We now consider  \eqref{realproduct} for $\psi_{j,\lam_j} \in V_{+}^2$ satisfying $\|\psi_{j,\lam_j}\|_{V_{+}^2}=1$ for $j=1,2$. Let us define the functions $\psi_p, \, \psi_m$ by $\psi_{p} := \psi_{1,\lam_1} + \psi_{2,\lam_2}$ $\psi_{m} := \psi_{1,\lam_1} - \psi_{2,\lam_2}$. In addition we readily get $P_{\sim \lam_1} \psi_{p} =  \psi_{p},\; P_{\sim \lam_1} \psi_{m} =  \psi_{m}$ and $ \|\psi_{p}\|_{V_{+}^2} \les 1, \; \|\psi_{m}\|_{V_{+}^2}\les 1$. For $c\in \mathbb{C}^2$ let ${\rm Re} \,c$ and ${\rm Im} \,c$ denote $\left( \begin{aligned}[l]  {\rm Re} \,c_1 \\ {\rm Re}\, c_2 \end{aligned} \right)$ and $\left( \begin{aligned}[l]  {\rm Im} \,c_1 \\ {\rm Im}\, c_2 \end{aligned} \right)$, respectively. Then we see that
	\begin{align*}
	{\rm Re}\,(\left< \psi_{1,\lam_1}, \psi_{2,\lam_2} \right>) &= \frac14 \Big(\left< {\rm Re}\, \psi_p , {\rm Re}\,\psi_p \right> - \left<{\rm Re}\, \psi_m , {\rm Re}\, \psi_m \right> + \left< {\rm Im}\, \psi_p , {\rm Im}\, \psi_p \right> - \left< {\rm Im} \,\psi_m , {\rm Im} \,\psi_m \right>\Big),\\
	{\rm Im}\, (\left< \psi_{1,\lam_1}, \psi_{2,\lam_2} \right>) &= { \rm Re}\,(-i\left< \psi_{1,\lam_1}, \psi_{2,\lam_2} \right>)
	\end{align*}
	which implies
	\begin{align}\label{appen1}
	\|P_{\le\mu} \left< \psi_{1,\lam_1}, \psi_{2,\lam_2} \right>\| \les \|P_{\le\mu} {\rm Re}\left< \psi_{1,\lam_1}, \psi_{2,\lam_2} \right>\| + \|P_{\le\mu} {\rm Re} \left< -i\psi_{1,\lam_1}, \psi_{2,\lam_2} \right>\|.
	\end{align}
	Since $P_{\sim \lam_1} ({\rm Re}\,\psi_p )=  \psi_p $ and $P_{\sim \lam_1} ({\rm Im}\,\psi_\pm) = P_{\sim \lam_1}\psi_\pm $, we obtain
	\begin{align*}
	\|P_{\le\mu} {\rm Re} \,\left< \psi_{1,\lam_1}, \psi_{2,\lam_2} \right>\| &\les \|P_{\le\mu}\left< {\rm Re} \,\psi_p, {\rm Re} \psi_p \right>\| + \|P_{\le\mu}\left< {\rm Re}\, \psi_m, {\rm Re}\, \psi_m \right>\|\\
	&\qquad\qquad + \|P_{\le\mu}\left< {\rm Im}\,\psi_p, {\rm Im}\psi_p \right>\| + \|P_{\le\mu}\left< {\rm Im}\,\psi_m, {\rm Im}\,\psi_m \right>\|\\
	&\les \mu^{\frac12} \left( \| {\rm Re}\, \psi_p\|_{V_{\pm}^2}^2  + \|{\rm Re}\,  \psi_m\|_{V_{\pm}^2}^2 + \|{\rm Im}\,\psi_p\|_{V_{\pm}^2}^2 + \|{\rm Im}\,\psi_m\|_{V_{\pm}^2}^2    \right)\\
	&\les \mu^{\frac12} \left( \|\psi_p\|_{V_{\pm}^2}^2  + \|\psi_m\|_{V_{\pm}^2}^2 \right) \les \mu^{\frac12}.
	\end{align*}
	where in second inequality we used the estimate \eqref{claim-bilinear}. The second term of right-hand side of \eqref{appen1} is obtained similarly as above.
		Therefore we get
	$$
		\left\|P_\mu \left< \psi_{1,\lam_1}, \psi_{2,\lam_2} \right> \right\| \les \mu^{\frac12}\|\psi_{1,\lam_1}\|_{V_{+}^2} \|\psi_{2,\lam_2}\|_{V_{+}^2}
	$$
	for any $\psi_{j,\lam_j}\in V_+^2$. This completes the proof of \eqref{eq-bilinear2}.
\end{proof}

By Propositions \ref{non-resonance} and \ref{resonance} we reach the following corollary.
\begin{cor}\label{bilinear2} Let  $\mu,\lam_1,\lam_2 \in 2^{\mathbb{Z}}$ satisfy $\mu \ll \lam_1 \sim \lam_2$, $0<\mu \le 1$, and $\lam_1,\lam_2 \ge 1$.  Assume that $\psi_{j,\lam_j} = \Pi_{\pm_j}(D)\psi_{j,\lam_j} \in V_{\pm_j}^2$ for $j = 1,2$. Then
	\begin{align}
	\left\|P_\mu \left< \psi_{1,\lam_1}, \beta\psi_{2,\lam_2} \right> \right\| \les \mu^{\frac12}\|\psi_{1,\lam_1}\|_{V_{\pm_1}^2} \|\psi_{2,\lam_2}\|_{V_{\pm_2}^2}.
	\end{align}
\end{cor}

\section{Proof of Theorem \ref{mainthm}}

\newcommand{\x}{X_{\pm}^s}


In this section we prove Theorem \ref{mainthm} by contraction argument. At first we denote time-restricted space $\left\{ \psi = \chi_{\left[0, \infty\right)} \phi : \phi \in U_{\pm}^2\right\}$ by $U_{\pm}^2(\left[0, \infty\right))$. And let us define Banach spaces $\x, X^s$ as
$$
\x := \left\{ \psi \in C(\mathbb{R} ; L_x^2)\cap U_{\pm}^2(\left[0, \infty\right)) : \|\psi\|_{\x} := \left( \sum_{\lam\ge1} \lam^{2s}\|P_\lam^1 \psi\|_{U_{\pm}^2}^2\right)^{\frac12} < \infty \right\},
$$
 \begin{align*}
X^s := \left\{ \psi : \psi_\pm : =\Pi_{\pm}(D)\psi \in \x  \;\;\mbox{and}\;\;  \|\psi\|_{X^s}:= \|\psi_+\|_{X_+^s}  +  \|\psi_-\|_{X_-^s} < \infty   \right\}.
\end{align*}
Let us consider a complete metric space $(X^s(\delta), d)$ and a map $\mathcal{H}$ on $X^s(\de)$ given by
 \begin{align*}
 X^s(\de) = \left\{ \psi \in X^s : \|\psi\|_{X^s} \le \de  \right\}  \;\;\mbox{with metric}\;\;  d(\psi, \phi) := \|\psi - \phi \|_{X^s},
 \end{align*}
\begin{align}\label{mapping}
\mathcal{H}(\psi) &=  \sum_{\pm_0 \in \{ \pm\}}e^{\mp_0 it\brad}\Pi_{\pm_0}(D)\psi_0 + i \sum_{\pm_j \in \{\pm\}} \mathcal{N}_{\pm_0}(\psi_{\pm_1},\psi_{\pm_2},\psi_{\pm_3})(t),
\end{align}
where
\begin{align*}
\mathcal{N}_{\pm_0}(\psi_{\pm_1},\psi_{\pm_2},\psi_{\pm_3})(t) = \int_0^t e^{\mp_0 i(t-t')\brad}\Pi_{\pm_0}(D)[(V*\left< \psi_{\pm_1}, \beta\psi_{\pm_2}\right>)\beta\psi_{\pm_3}] d\,t'.
\end{align*}
Then we will show that $\mathcal{H}$ is a contraction mapping on $X^s(\de)$.
Indeed, the linear part of \eqref{mapping} can be handled as follows:
\begin{align}\label{linear-est}
\left\|e^{\mp_0 it\brad} \Pi_{\pm_0}(D) \psi_0 \right\|_{X_{\pm_0}^s}^2 = \sum_{\lam \ge 1} \lambda^{2s}\left\|\chi_{[0,\infty)} P_\lam^1 \Pi_{\pm_0}(D) \psi_0 \right\|_{U_{\pm_0}^2}^2 \le C\|\psi_0\|_{H^s}^2.
\end{align}
For the nonlinear parts $\mathcal{N}_{\pm_0}(\psi_{\pm_1},\psi_{\pm_2},\psi_{\pm_3})$ we will show the following trilinear estimates.
\begin{prop}\label{nonlinear}
	Let $\psi_{j,\pm_j} \in X_{\pm_j}^s$ for $s > \gamma-1$. Then we have
	\begin{align*}
	\|\mathcal{N}_{\pm_0}(\psi_{1,\pm_1},\psi_{2,\pm_2},\psi_{3,\pm_3})\|_{X_{\pm_0}^s} \les \prod_{j=1}^3\|\psi_{j,\pm_j}\|_{X_{\pm_j}^s}.
	\end{align*}
\end{prop}

Once Proposition \ref{nonlinear} has been proved, this trilinear estimate together with linear estimate \eqref{linear-est} yields
\begin{align*}
\|\mathcal{H}(\psi)\|_{X^s} = \sum_{\pm_0 \in \{\pm\}}\|\Pi_{\pm_0}\mathcal H(\psi)\|_{X_{\pm_0}^s} \le C(\|\psi_0\|_{H^s} + \|\psi\|_{X^s}^3).
\end{align*}
If $\de$ is small enough so that $C\delta^3 \le \frac\delta8$ and $\psi_0$ satisfies $C\|\psi_0\|_{H^s} \le \frac\delta2$,
$\mathcal H$ is a self-mapping on $X^s(\delta)$. In particular, we get
\begin{align*}
\|\mathcal{H}(\psi) - \mathcal{H}(\phi)\|_{X^s} &\le C\left(\|\psi\|_{X^s}+ \|\phi\|_{X^s}\right)^2 \|\psi- \phi\|_{X^s} \le 4C\de^2\|\psi- \phi\|_{X^s} \le \frac12\|\psi- \phi\|_{X^s}.
\end{align*}
Hence $\mathcal H: X^s(\de) \to X^s(\de)$ is a contraction mapping for sufficiently small $\de$ and then we get a unique solution $\psi_p \in L^\infty([0, \infty); H^s)$ to \eqref{maineq}. The time continuity and continuous dependency on data follow readily from the Duhamel's formula and Proposition \ref{nonlinear}. By the time symmetry of \eqref{maineq} we also obtain a unique solution $\psi_n \in C((-\infty, 0], H^s)$ with the continuous dependency on data. Defining $\psi = \psi_p + \psi_n$, we get the global well-posedness of \eqref{maineq}.

Now we move onto the scattering property of \eqref{maineq-decom}. Since the backward scattering can be treated similarly to the forward one, we omit its proof. For $\lambda \ge 1$ let us define
$$
\vp_{\lam,\pm}:=  P_\lambda^1 \epm  \mathcal{N}_{\pm}(\psi),
$$
where $\mathcal{N}_{\pm}(\psi) = \sum_{\pm_j \in \{\pm\}} \mathcal{N}_{\pm}(\psi_{\pm_1},\psi_{\pm_2},\psi_{\pm_3})(t)$. Then Lemma \ref{embedd} shows that
$$
\vp_{\lam,\pm} \in V_{\pm}^2
$$
for all $\lam\ge1$. Since $\sum_{\lam \ge1} \lam^{2s} \|\vp_{\lam,\pm}\|_{V_\pm^2} \les 1$, we have
$$
\phi_\pm :=\psi_{0,\pm} + \lim_{t \to \infty} \sum_{\lam\ge1}\vp_{\lam,\pm} \in H^s
$$
and
$$
\|\psi_\pm(t) - \emp \phi_\pm\|_{H^s}  \xrightarrow{t \to \infty} 0.
$$
This completes the proof of scattering of \eqref{maineq-decom}.
\subsection {Proof of Proposition \ref{nonlinear}}
By Lemma \ref{duality} we obtain
\begin{align*}
& \left\|P_{\lam_4}^1 \int_0^t e^{\mp i(t-t')\brad}\Pi_{\pm}(D)[(V*\left< \psi_{1,\pm_1}, \beta\psi_{2,\pm_2}\right>)\beta\psi_{3,\pm_3}] dt' \right\|_{U_{\pm\brad}^2}\\
& \qquad\qquad =\left\|P_{\lam_4}^1 \int_0^t e^{\pm it'\brad}\Pi_{\pm}(D)[(V*\left< \psi_{1,\pm_1}, \beta\psi_{2,\pm_2}\right>)\beta\psi_{3,\pm_3}] dt' \right\|_{U^2}\\
& \qquad\qquad = \sup_{\begin{subarray}{ll} \|\phi\|_{V^2}=1 \\ \;\;\phi \in C_0^\infty  \end{subarray}}\left| \iint (V*\left< \psi_{1,\pm_1}, \beta\psi_{2,\pm_2}\right>)\left<\beta\psi_{3,\pm_3}, \Pi_{\pm}(D)P_{\lam_4}^1 e^{\mp it\brad}\phi\right> dt dx\right|\\
& \qquad\qquad = \sup_{\|\psi_4\|_{V_{\pm\brad}^2}=1}\left| \iint (V*\left< \psi_{1,\pm_1}, \beta\psi_{2,\pm_2}\right>)\left<\beta\psi_{3,\pm_3}, \Pi_{\pm}(D)P_{\lam_4}^1 \psi_4\right> dt dx\right|.
\end{align*}
This induces that
\begin{align*}
\|\mathcal{N}_{\pm}(\psi_{1,\pm_1},\psi_{2,\pm_2},\psi_{3,\pm_3})\|_{X_\pm^s}^2 \les \sum_{ \lam_4 \ge 1} \lam_4^{2s} \left( \sup_{\|\psi_4\|_{V_{\pm\brad}^2}=1}\left| \sum_{\lam_i \ge 1}\iint (V*\left< \psi_{1,\pm_1}, \beta\psi_{2,\pm_2}\right>)\left<\beta\psi_{3,\pm_3},  \psi_{4,\pm_4}\right> dt dx\right| \right)^2.
\end{align*}
The following lemma to be proved in the next section plays a crucial role in the proof of Proposition \ref{nonlinear}.
\begin{lem}\label{non-esti} Let $\de > \gam -1$. Then we have for all $\psi_{j,\pm_j}\in U_{\pm_j}^2,\; \psi_{4,\pm_4} \in V_{\pm_4}^2$
	\begin{align}\label{non-integral}
	\left| \iint (V*\left< \psi_{1,\pm_1}, \beta\psi_{2,\pm_2}\right>)\left<\beta\psi_{3,\pm_3},  \psi_{4,\pm_4}\right> dt dx\right| \les (\lammed)^\de \prod_{j=1}^{3}\|\psi_{j,\pm_j}\|_{U_{\pm_j}^2}\|\psi_{4,\pm_4}\|_{V_{\pm_4}^2},
	\end{align}
	where $\lammed$ is median of $\lam_1,\lam_2,\lam_3$.
\end{lem}
For the sake of simplicity we denote $\psi_{\pm_j}$ by $\psi_j$. Then Lemma \ref{non-esti} implies that for $\gam -1 <\de < s$
\begin{align*}
\|\mathcal{N}_{\pm}(\psi_{1},\psi_{2},\psi_{3})\|_{X_\pm^s}^2 \les \sum_{ \lam_4 \ge 1} \lam_4^{2s} \left( \sum_{\lam_i \ge 1} (\lammed)^{\de} \prod_{j=1}^{3}\|\psi_j\|_{U_{\pm_j}^2}\right)^2.
\end{align*}
We divide the cases into (i) $\lam_3 \sim \lam_4$, (ii) $\lam_3 \ll \lam_4$, and (iii) $\lam_3 \gg \lam_4$.
\paragraph{(i) The case: $\lam_3 \sim \lam_4$} By Cauchy-Schwarz we obtain
\begin{align*}
\|\mathcal{N}_{\pm}(\psi_{1},\psi_{2},\psi_{3})\|_{X_\pm^s}^2 &\les \sum_{ \lam_4 \ge 1} \lam_4^{2s} \left( \sum_{\lam_i \ge 1} (\lammed)^{\de} \prod_{j=1}^{3}\|\psi_j\|_{U_{\pm_j}^2}\right)^2\\
                                                    &\les \sum_{ \lam_4 \ge 1} \lam_4^{2s}  \left(\sum_{\lam_3 \sim \lam_4} \lam_1^{\de-s}\lam_2^{\de-s} \lam_1^s\|\psi_1\|_{U_{\pm_1}^2}\lam_2^s\|\psi_2\|_{U_{\pm_2}^2}\|\psi_3\|_{U_{\pm_3}^2} \right)^2\\
                                                    &\les \|\psi_1\|_{X_{\pm_1}^s}^2\|\psi_2\|_{X_{\pm_2}^s}^2\sum_{ \lam_4 \ge 1} \left(\sum_{\lam_3 \sim \lam_4} \lam_4^{s}\|\psi_3\|_{U_{\pm_3}^2}\right)^2\\
                                                    &\les \prod_{j=1}^{3}\|\psi_{j}\|_{X_{\pm_j}^s}^2.
\end{align*}

\paragraph{(ii) The case: $\lam_3 \ll \lam_4$} We further divide the case into $\lam_1 \sim \lam_2$ and $\lam_1 \nsim \lam_2$.
If $\lam_1 \sim \lam_2$, by Cauchy-Schwarz inequality and Young's convolution inequality  we get
\begin{align*}
\|\mathcal{N}_{\pm}(\psi_{1},\psi_{2},\psi_{3})\|_{X_\pm^s}^2 &\les \sum_{ \lam_4 \ge 1} \lam_4^{2s} \left( \sum_{\lam_i \ge 1} (\lammed)^{\de} \prod_{j=1}^{3}\|\psi_j\|_{U_{\pm_j}^2}\right)^2\\
&\les \sum_{ \lam_4 \ge 1} \lam_4^{2s}  \left(\sum_{\lam_3 \ll \lam_4} \lam_1^{\de-s}\lam_3^{-s} \lam_1^s\|\psi_1\|_{U_{\pm_1}^2}\|\psi_2\|_{U_{\pm_2}^2}  \lam_3^s\|\psi_3\|_{U_{\pm_3}^2} \right)^2\\
&\les \|\psi_1\|_{X_{\pm_1}^s}^2\|\psi_3\|_{X_{\pm_3}^s}^2\sum_{ \lam_4 \ge 1} \left(\sum_{\lam_4 \ll \lam_2} \left(\frac{\lam_4}{\lam_2}\right)^{s} \lam_2^s\|\psi_2\|_{U_{\pm_2}^2}\right)^2\\
&\les \prod_{j=1}^{3}\|\psi_{j}\|_{X_{\pm_j}^s}^2.
\end{align*}
Next if $\lam_1 \ll \lam_2$, we have
\begin{align*}
\|\mathcal{N}_{\pm}(\psi_{1},\psi_{2},\psi_{3})\|_{X_\pm^s}^2 &\les \sum_{ \lam_4 \ge 1} \lam_4^{2s} \left( \sum_{\lam_i \ge 1} (\lammed)^{\de} \prod_{j=1}^{3}\|\psi_j\|_{U_{\pm_j}^2}\right)^2\\
&\les \sum_{ \lam_4 \ge 1} \lam_4^{2s}  \left(\sum_{\lam_3 \ll \lam_4} \lam_1^{\de-s}\lam_3^{\de-s} \lam_1^s\|\psi_1\|_{U_{\pm_1}^2}\|\psi_2\|_{U_{\pm_2}^2}  \lam_3^s\|\psi_3\|_{U_{\pm_3}^2} \right)^2\\
&\les \|\psi_1\|_{X_{\pm_1}^s}^2\|\psi_3\|_{X_{\pm_3}^s}^2\sum_{ \lam_4 \ge 1} \left(\sum_{\lam_4 \sim \lam_2}  \lam_2^s\|\psi_2\|_{U_{\pm_2}^2}\right)^2\\
&\les \prod_{j=1}^{3}\|\psi_{j}\|_{X_{\pm_j}^s}^2.
\end{align*}
In the case $\lam_1 \gg \lam_2$, a similar estimate to the above can be obtained.

\smallskip

\paragraph{(iii) The case: $\lam_3 \gg \lam_4$} We also divide this case into $\lam_1 \sim \lam_2$ and $\lam_1 \nsim \lam_2$.
If $\lam_1 \sim \lam_2$, by Cauchy-Schwarz inequality and Young's convolution inequality we get
\begin{align*}
\|\mathcal{N}_{\pm}(\psi_{1},\psi_{2},\psi_{3})\|_{X_\pm^s}^2 &\les \sum_{ \lam_4 \ge 1} \lam_4^{2s} \left( \sum_{\lam_i \ge 1} (\lammed)^{\de} \prod_{j=1}^{3}\|\psi_j\|_{U_{\pm_j}^2}\right)^2\\
&\les \sum_{ \lam_4 \ge 1} \lam_4^{2s}  \left(\sum_{\lam_3 \gg \lam_4} \lam_1^{\de-s}\lam_2^{\de-s} \lam_1^s\|\psi_1\|_{U_{\pm_1}^2}\lam_2^s\|\psi_2\|_{U_{\pm_2}^2}\|\psi_3\|_{U_{\pm_3}^2} \right)^2\\
&\les \|\psi_1\|_{X_{\pm_1}^s}^2\|\psi_2\|_{X_{\pm_2}^s}^2\sum_{ \lam_4 \ge 1} \left(\sum_{\lam_3 \gg \lam_4} \left( \frac{\lam_4}{\lam_3}\right)^s \lam_3^{s}\|\psi_3\|_{U_{\pm_3}^2}\right)^2\\
&\les \prod_{j=1}^{3}\|\psi_{j}\|_{X_{\pm_j}^s}^2.
\end{align*}
Next if $\lam_1 \ll \lam_2$, we have
\begin{align*}
\|\mathcal{N}_{\pm}(\psi_{1},\psi_{2},\psi_{3})\|_{X_\pm^s}^2 &\les \sum_{ \lam_4 \ge 1} \lam_4^{2s} \left( \sum_{\lam_i \ge 1} (\lammed)^{\de} \prod_{j=1}^{3}\|\psi_j\|_{U_{\pm_j}^2}\right)^2\\
&\les \sum_{ \lam_4 \ge 1} \lam_4^{2s}  \left(\sum_{\lam_3 \gg \lam_4} \lam_1^{\de-s}\lam_2^{\de-s} \lam_1^s\|\psi_1\|_{U_{\pm_1}^2}\lam_2^s\|\psi_2\|_{U_{\pm_2}^2}  \|\psi_3\|_{U_{\pm_3}^2} \right)^2\\
&\les \|\psi_1\|_{X_{\pm_1}^s}^2\|\psi_3\|_{X_{\pm_3}^s}^2\sum_{ \lam_4 \ge 1} \left(\sum_{\lam_3 \gg \lam_4}  \left( \frac{\lam_4}{\lam_3}\right)^s \lam_3^s\|\psi_3\|_{U_{\pm_3}^2}\right)^2\\
&\les \prod_{j=1}^{3}\|\psi_{j}\|_{X_{\pm_j}^s}^2.
\end{align*}
The case $\lam_1 \gg \lam_2$ can be treated similarly as above. Therefore we get the desired results.

\section{Proof of Lemma \ref{non-esti}}
\newcommand{\pmu}{P_{\mu}}
\newcommand{\utwo}{U_{\pm}^2}
\newcommand{\utwoj}{U_{\pm_j}^2}
\newcommand{\vtwo}{V_{\pm}^2}
\newcommand{\vtwoj}{V_{\pm_j}^2}
\newcommand{\ufour}{U_{\pm}^4}
\newcommand{\ufourj}{U_{\pm_j}^4}
In this section we prove the Lemma \ref{non-esti}. We decompose the left-hand side of \eqref{non-integral} into low-high frequency parts as follows:
\begin{align*}
J({\lam_1,\lam_2,\lam_3,\lam_4}) &:= \left| \iint V*\left< \psi_1, \beta\psi_2\right>\left<\beta\psi_3, \psi_4\right>  dxdt  \right|\\
                                 &= \sum_{\mu >0}  \left| \iint P_\mu \left( V*\left< \psi_1, \beta\psi_2\right>\right)\cdot  \left( \widetilde{P_\mu} \left<\beta\psi_3, \psi_4\right>  \right) dxdt  \right|\\
                                 &= \sum_{0<\mu \le 1}\left| \iint P_\mu \left( V*\left< \psi_1, \beta\psi_2\right>\right)\cdot  \left( \widetilde{P_\mu} \left<\beta\psi_3, \psi_4\right>  \right) dxdt  \right| \\
                                 & \qquad  +  \sum_{\mu > 1}\left| \iint P_\mu \left( V*\left< \psi_1, \beta\psi_2\right>\right)\cdot  \left( \widetilde{P_\mu} \left<\beta\psi_3, \psi_4\right>  \right) dxdt  \right|\\
                                 &=: J^l + J^h.
\end{align*}

This section is composed of three subsections. In the first one we deal with the low-frequency parts by using fully the estimates of Section \ref{low-esti}. In the next one we reveal the modulation lemma useful to treat high-frequency parts. And the detailed proof of high-frequency parts is provided in the third subsection. From now on we may assume that $\lam_1 \le \lam_2$, $\lam_3 \le \lam_4$, and $\pm_1 = \pm_3 = +$ thanks to the symmetry.

\subsection{Low-frequency parts}\label{lowfrequency}
Since $0 < \mu \le 1$ and $\lam_j\ge1(j=1,\cdots,4)$, the support condition yields that $\lam_1 \sim \lam_2$ and $\lam_3 \sim \lam_4$. To prove Lemma \ref{non-esti} for $J^l$ we use the Corollary \ref{bilinear2}. In fact, we get
\begin{align*}
J^l &= \sum_{0<\mu \le 1}\left| \iint P_\mu \left( V*\left< \psi_1, \beta\psi_2\right>\right)\cdot  \left( \widetilde{P_\mu} \left<\beta\psi_3, \psi_4\right>  \right) dxdt  \right| \\
&\les  \sum_{0<\mu \le 1} \left\| P_\mu \left( V*\left< \psi_1, \beta\psi_2\right>\right) \right\| \left\| \widetilde{P_\mu} \left<\beta\psi_3, \psi_4\right>  \right\|\\
&\les \sum_{0<\mu \le 1} \mu^{-1+ \gam} \prod_{j=1}^3\|\psi_j\|_{V_{\pm_j}^2}\|\psi_4\|_{V_{\pm_4}^2}\\
&\les \prod_{j=1}^3\|\psi_j\|_{V_{\pm_j}^2}\|\psi_4\|_{V_{\pm_4}^2}.
\end{align*}

\subsection{Modulation estimates}
Let us define the modulation projections by
$$
\Lambda_{\lam}^\pm u = \mathcal{F}_{t,x}^{-1} \left( \beta_\lam(|\tau \pm \left< \xi \right> |) \mathcal{F}_{t,x} u \right).
$$
and
$$
\Lambda_{\ge \lam}^\pm u = \sum_{\sigma \ge \lam}\Lambda_{\sigma}^\pm u, \qquad  \Lambda_{<\lam}^\pm u = 1- \Lambda_{\ge\lam}^\pm u.
$$

\begin{lem}[Corollary 2.18 of \cite{haheko}]\label{modul} Let $\lam \in 2^{\mathbb{Z}}$. Then
$$
\| \Lambda_{\ge \lam}^\pm u \| \les \lam^{-\frac12}\|u\|_{\vtwo}.
$$
\end{lem}

\begin{lem}
Let $\pm_4=-$. Suppose $\lam_j \in 2^{\mathbb{Z}}$ satisfy $\lam_1 \le \lam_2 \ll \lam_3 \sim \lam_4$ and $|\xi| \ge 1$, where $\xi$ is the frequency of $V$. Then
\begin{align}
\left| \iint V* B(\psi_1,\psi_2)\cdot Q(\psi_3,\psi_4)  dxdt  \right| \les \lam_2^{\de}   \prod_{j=1}^3\|\psi_j\|_{U_{\pm_j}^2}\|\psi_4\|_{V_{\pm_4}^2},  \label{j3-esti}\\
\left| \iint V* Q(\psi_1,\psi_2)\cdot Q(\psi_3,\psi_4)  dxdt  \right| \les \lam_2^{\de}   \prod_{j=1}^3\|\psi_j\|_{U_{\pm_j}^2}\|\psi_4\|_{V_{\pm_4}^2}, \label{j4-esti}
\end{align}
for any $\de >0$.
\end{lem}
\begin{proof}
To prove  \eqref{j3-esti} and \eqref{j4-esti} we decompose $\psi_j$ into  low and high modulation parts
$$
\psi_j = \Lambda_{<\frac{\lam_3}8}^{\pm_j} \psi_j + \Lambda_{\ge\frac{\lam_3}8}^{\pm_j} \psi_j =: \psi_j^l + \psi_j^h.
$$
 Since $\tau_1 - \tau_2 + \tau_3 -\tau_4 =0$ and $\xi_1 - \xi_2 + \xi_3 -\xi_4 =0$, using $\pm_1 = \pm_3 = +,\pm_4 =-$, we have
$$
\left| \iint V* Y(\psi_1^l,\psi_2^l)\cdot Q(\psi_3^l,\psi_4^l)  dxdt  \right| = 0,
$$
where $Y$ is the one of $B$ or $Q$. Therefore we do not consider the case $\psi_j = \psi_j^l$ for all $j = 1,\cdots,4$.

We first prove the \eqref{j3-esti}. By Bernstein's and H\"older's inequalities and Lemma \ref{null-esti}, we have
\begin{align*}
\left| \iint V* B(\psi_1^h,\psi_2)\cdot Q(\psi_3,\psi_4)  dxdt  \right| &\les \sum_{\mu \ge 1}\left\|\pmu \left(V*B(\psi_1^h,\psi_2)\right) \right\|_{L_t^2 L_x^{\frac1{1-\ve}}}  \left\| \widetilde{\pmu}Q(\psi_3,\psi_4)\right\|_{L_t^2L_x^{\frac1\ve}}\\
&\les  \sum_{\mu \ge 1} \mu ^{-1 +\gam -2\ve}\|\psi_1^h\|_{L_{t,x}^2}   \|\psi_2\|_{L_t^\infty L_x^\frac2{1-2\ve}}  \left\| \widetilde{\pmu}Q(\psi_3,\psi_4)\right\|_{L_{t,x}^2}\\
&\les  \sum_{\mu \ge 1}  \mu^{-\frac12+\gam -\de-2\ve} \lam_2^{\de +2\ve} \lam_3^{-\frac12} \prod_{j=1}^3\|\psi_j\|_{U_{\pm_j}^2}\|\psi_4\|_{U_{\pm_4}^2}\\
&\les  \lam_2^{\de} \lam_3^{-\frac{\de+1 -\gam}2}  \prod_{j=1}^3\|\psi_j\|_{U_{\pm_j}^2}\|\psi_4\|_{U_{\pm_4}^2}
\end{align*}
for some $0< \ve \ll1$. By \eqref{endpoint}, we also obtain
\begin{align*}
\left| \iint V* B(\psi_1^h,\psi_2)\cdot Q(\psi_3,\psi_4)  dxdt  \right| &\les \sum_{\mu \ge 1}\left\|\pmu \left(V*B(\psi_1^h,\psi_2)\right) \right\|  \left\| \widetilde{\pmu}Q(\psi_3,\psi_4)\right\|\\
&\les \sum_{\mu \ge 1}  \mu^{-2+\gam}   \lam_3  \prod_{j=1}^3\|\psi_j\|_{U_{\pm_j}^2}\|\psi_4\|_{U_{\pm_4}^4}\\
&\les  \lam_3  \prod_{j=1}^3\|\psi_j\|_{U_{\pm_j}^2}\|\psi_4\|_{U_{\pm_4}^4}.
\end{align*}
Applying Lemma \ref{log} to the above two estimates, we get
$$
\left| \iint V* B(\psi_1^h,\psi_2)\cdot Q(\psi_3,\psi_4)  dxdt  \right| \les  \lam_2^{\de}   \prod_{j=1}^3\|\psi_j\|_{U_{\pm_j}^2}\|\psi_4\|_{V_{\pm_4}^2}.
$$
If $\psi_2 = \psi_2^h$, we may obtain the desired result in a similar manner.

If $\psi_3 = \psi_3^h$, then we get
\begin{align*}
\left| \iint V* B(\psi_1,\psi_2)\cdot Q(\psi_3^h,\psi_4)  dxdt  \right| &\les \sum_{\mu \ge 1}\left\|\pmu \left(V*B(\psi_1,\psi_2)\right) \right\|_{L_t^2 L_x^\infty}  \left\| \widetilde{\pmu}Q(\psi_3^h,\psi_4)\right\|_{L_t^2L_x^1}\\
&\les  \sum_{\mu \ge 1} \mu^{-1+\gam}\left\| \pmu B(\psi_1,\psi_2)\right\|_{L_{t,x}^2}  \|\psi_3^h\|_{L_{t,x}^2}   \|\psi_4\|_{L_t^\infty L_x^2} \\
&\les  \sum_{\mu \ge 1}  \mu^{-1+\gam}  \lam_3^{-\frac12} \prod_{j=1}^3\|\psi_j\|_{U_{\pm_j}^2}\|\psi_4\|_{V_{\pm_4}^2}\\
&\les  \lam_2^{\de} \sum_{\mu \ge 1} \mu^{-\frac32+\gam -\de}  \prod_{j=1}^3\|\psi_j\|_{U_{\pm_j}^2}\|\psi_4\|_{V_{\pm_4}^2}.
\end{align*}
The case $\psi_4 = \psi_4^h$ can also be treated similarly.

We next consider the \eqref{j4-esti}. The proof of \eqref{j4-esti} is very similar to the one of \eqref{j3-esti} except changing $B(\psi_1, \psi_2)$ with $Q(\psi_1,\psi_2)$. For the convenience of readers we brief on the details. By the consecutive use of Bernstein's inequality, Lemma \ref{null-esti}, and Lemma \ref{modul} we have
\begin{align}
\begin{aligned}\label{qq-u2}
\left| \iint V* Q(\psi_1^h,\psi_2)\cdot Q(\psi_3,\psi_4)  dxdt  \right| &\les \sum_{\mu \ge 1}\left\|\pmu \left(V*Q(\psi_1^h,\psi_2)\right) \right\|_{L_t^2 L_x^{\frac1{1-\ve}}}  \left\| \widetilde{\pmu}Q(\psi_3,\psi_4)\right\|_{L_t^2L_x^{\frac1\ve}}\\
&\les  \sum_{\mu \ge 1} \mu^{-1+\gam -2\ve}  \|\psi_1^h\|_{L_{t,x}^2}   \|\psi_2\|_{L_t^\infty L_x^{\frac2{1-2\ve}}}  \left\| \widetilde{\pmu}Q(\psi_3,\psi_4)\right\|_{L_{t,x}^2}\\
&\les  \sum_{\mu \ge 1}  \mu^{-\frac12+\gam-\de-2\ve} \lam_2^{\de + 2\ve} \lam_3^{-\frac12} \prod_{j=1}^3\|\psi_j\|_{U_{\pm_j}^2}\|\psi_4\|_{U_{\pm_4}^2}\\
&\les  \lam_2^{\de} \lam_3^{-\frac{\de+1 -\gam}2}  \prod_{j=1}^3\|\psi_j\|_{U_{\pm_j}^2}\|\psi_4\|_{U_{\pm_4}^2}
\end{aligned}
\end{align}
for some $0 <\ve \ll 1$. In a similar way we get
\begin{align}
\begin{aligned}\label{qq-u4}
\left| \iint V* Q(\psi_1^h,\psi_2)\cdot Q(\psi_3,\psi_4)  dxdt  \right|  &\les  \sum_{\mu \ge 1}  \mu^{-1+\gam -2\ve} \lam_2^{2\ve}\|\psi_1^h\|_{L_{t,x}^2}   \|\psi_2\|_{L_t^\infty L_x^2}  \left\| \widetilde{\pmu}Q(\psi_3,\psi_4)\right\|_{L_{t,x}^2}\\
&\les \sum_{\mu \ge 1} \mu^{-1+\gam - 2\ve}    \lam_3^{\frac12 + 2\ve}  \prod_{j=1}^3\|\psi_j\|_{U_{\pm_j}^2}\|\psi_4\|_{U_{\pm_4}^4}\\
&\les   \lam_3^{\frac12 + 2\ve}  \prod_{j=1}^3\|\psi_j\|_{U_{\pm_j}^2}\|\psi_4\|_{U_{\pm_4}^4}.
\end{aligned}
\end{align}
The logarithm interpolation(Lemma \ref{log}) between \eqref{qq-u2} and \eqref{qq-u4} gives us that
$$
\left| \iint V* Q(\psi_1^h,\psi_2)\cdot Q(\psi_3,\psi_4)  dxdt  \right| \les  \lam_2^{\de}   \prod_{j=1}^3\|\psi_j\|_{U_{\pm_j}^2}\|\psi_4\|_{V_{\pm_4}^2}.
$$
The case $\psi_2 = \psi_2^h$ can be estimated similarly.

We now treat the case $\psi_3 = \psi_3^h$ as follows.
\begin{align*}
\left| \iint V* Q(\psi_1,\psi_2)\cdot Q(\psi_3^h,\psi_4)  dxdt  \right| &\les \sum_{\mu \ge 1}\left\|\pmu \left(V*Q(\psi_1,\psi_2)\right) \right\|_{L_t^2 L_x^\infty}  \left\| \widetilde{\pmu}Q(\psi_3^h,\psi_4)\right\|_{L_t^2L_x^1}\\
&\les  \sum_{\mu \ge 1} \mu^{-1+\gam} \left\| \pmu Q(\psi_1,\psi_2)\right\|_{L_{t,x}^2}  \|\psi_3^h\|_{L_{t,x}^2}   \|\psi_4\|_{L_t^\infty L_x^2} \\
&\les  \sum_{\mu \ge 1}  \mu^{-\frac12+\gam}    \lam_3^{-\frac12} \prod_{j=1}^3\|\psi_j\|_{U_{\pm_j}^2}\|\psi_4\|_{V_{\pm_4}^2}\\
&\les  \lam_2^{\de} \sum_{\mu \ge 1} \mu^{-1 +\gam -\de}  \prod_{j=1}^3\|\psi_j\|_{U_{\pm_j}^2}\|\psi_4\|_{V_{\pm_4}^2}.
\end{align*}
This completes the proof of \eqref{j4-esti}.
\end{proof}

\subsection{High-frequency parts}
For the high-frequency parts we utilize the null-form and bilinear estimates. To this end we further decompose the high-frequency parts of nonlinear term into bilinear terms as follows:

\begin{align*}
J^h &= \sum_{\mu\ge1}\left| P_\mu \left(\iint V*\left< \psi_1, \beta\psi_2\right>\right)\cdot  \left( \widetilde{P_\mu} \left<\beta\psi_3, \psi_4\right>  \right) dxdt  \right|\\
&\les J_1^h + J_2^h + J_3^h + J_4^h,
\end{align*}
where
\begin{align*}
J_1^h = \sum_{\mu\ge1} \left| \iint V* B(\psi_1,\psi_2)\cdot B(\psi_3,\psi_4)  dxdt\right|,\\
J_2^h = \sum_{\mu\ge1} \left| \iint V* Q(\psi_1,\psi_2)\cdot B(\psi_3,\psi_4)  dxdt\right|,\\
J_3^h = \sum_{\mu\ge1} \left| \iint V* B(\psi_1,\psi_2)\cdot Q(\psi_3,\psi_4)  dxdt\right|,\\
J_4^h = \sum_{\mu\ge1} \left| \iint V* Q(\psi_1,\psi_2)\cdot Q(\psi_3,\psi_4)  dxdt\right|.
\end{align*}

\subsubsection{ Estimates for $J_1^h$} By Lemma \ref{bilinear} and H\"older's inequality we get
\begin{align*}
J_1^h  &\les \sum_{\mu \ge 1}\mu^{-2+\gam}\left\|\pmu B(\psi_1,\psi_2)\right\| \left\| \widetilde{\pmu}B(\psi_3,\psi_4)\right\|   \les \sum_{\mu \ge 1}\mu^{-2+\gam} \left( \frac{\lam_2}{\lam_1}\right)^{\frac{2-\gam}4} \left( \frac{\lam_4}{\lam_3}\right)^{\frac{2-\gam}4} \prod_{j=1}^3\|\psi_j\|_{U_{\pm_j}^2}\|\psi_4\|_{V_{\pm_4}^2}\\
      &\les  \prod_{j=1}^3\|\psi_j\|_{U_{\pm_j}^2}\|\psi_4\|_{V_{\pm_4}^2}.
\end{align*}

\subsubsection{Estimates for $J_2^h$} It follows from Lemma \ref{null-esti} that
\begin{align*}
J_2^h &\les \sum_{\mu \ge 1}\mu^{-2+\gam}\left\|\pmu Q(\psi_1,\psi_2)\right\| \left\| \widetilde{\pmu}B(\psi_3,\psi_4)\right\|  \les \sum_{ \mu \ge 1} \mu^{-\frac32+\gam}  \left( \frac{\lam_4}{\lam_3}\right)^{\frac14} \prod_{j=1}^3\|\psi_j\|_{U_{\pm_j}^2}\|\psi_4\|_{V_{\pm_4}^2}\\
      &\les (\lammed)^{\de} \sum_{ \mu \ge 1} \mu^{-\frac32+\gam -\de}   \prod_{j=1}^3\|\psi_j\|_{U_{\pm_j}^2}\|\psi_4\|_{V_{\pm_4}^2} \les   (\lammed)^{\de} \prod_{j=1}^3\|\psi_j\|_{U_{\pm_j}^2}\|\psi_4\|_{V_{\pm_4}^2}.
\end{align*}
\subsubsection{Estimates for $J_3^h$} Let us consider the subcases: (i) $\lam_3 \ll \lam_4$, (ii) $\lam_3 \sim \lam_4 \les \lam_2$, and (iii) $\lam_2 \ll \lam_3 \sim \lam_4$.
\paragraph{(i) The case: $\lam_3 \ll \lam_4$}

By Lemma \ref{bilinear} we get
\begin{align*}
J_3^h &\les \sum_{\mu \ge 1}\mu^{-2+\gam}\left\|\pmu B(\psi_1,\psi_2)\right\| \left\| \widetilde{\pmu}Q(\psi_3,\psi_4)\right\| \les \sum_{\mu \ge 1}\mu^{-2+\gam} \left( \frac{\lam_2}{\lam_1}\right)^{\frac{2-\gam}4} \lam_3^{\frac{\de}2} \lam_4^{\frac{2-\gam}4}\prod_{j=1}^3\|\psi_j\|_{U_{\pm_j}^2}\|\psi_4\|_{V_{\pm_4}^2}\\
      &\les (\lam_3)^{\de}\prod_{j=1}^3\|\psi_j\|_{U_{\pm_j}^2}\|\psi_4\|_{V_{\pm_4}^2}    \les (\lammed)^{\de}\prod_{j=1}^3\|\psi_j\|_{U_{\pm_j}^2}\|\psi_4\|_{V_{\pm_4}^2}.
\end{align*}

\paragraph{(ii) The case: $\lam_3 \sim \lam_4 \les \lam_2$}
By Lemmas \ref{null-esti} and \ref{bilinear} we have
\begin{align}
\begin{aligned}\label{j3-u2}
J_3^h  &\les \sum_{\mu \ge 1}\mu^{-2+\gam}\left\|\pmu B(\psi_1,\psi_2)\right\| \left\| \widetilde{\pmu}Q(\psi_3,\psi_4)\right\|  \les  \sum_{\mu \ge 1}\mu^{-\frac32+\gam} \left( \frac{\lam_2}{\lam_1}\right)^{\frac{1}4} \prod_{j=1}^3\|\psi_j\|_{U_{\pm_j}^2}\|\psi_4\|_{U_{\pm_4}^2}\\
&\les (\lammed)^{\de} \sum_{\mu \ge 1}\mu^{-\frac54+\gam - \de}  \prod_{j=1}^3\|\psi_j\|_{U_{\pm_j}^2}\|\psi_4\|_{U_{\pm_4}^2}  \les   (\lammed)^{\de}\prod_{j=1}^3\|\psi_j\|_{U_{\pm_j}^2}\|\psi_4\|_{U_{\pm_4}^2}
\end{aligned}
\end{align}
and
\begin{align}
\begin{aligned}\label{j3-u4}
J_3^h &\les \sum_{\mu \ge1}\mu^{-1} \left( \frac{\lam_2}{\lam_1}\right)^{\de} \lam_3^\frac12 \lam_4^\frac12\prod_{j=1}^3\|\psi_j\|_{U_{\pm_j}^2}\|\psi_4\|_{U_{\pm_4}^4} \les \lam_3\prod_{j=1}^3\|\psi_j\|_{U_{\pm_j}^2}\|\psi_4\|_{U_{\pm_4}^4}.
\end{aligned}
\end{align}
Using logarithm interpolation(Lemma \ref{log}) between \eqref{j3-u2} and \eqref{j3-u4}, we get
$$
J_3^h \les (\lam_3)^{\de}\prod_{j=1}^3\|\psi_j\|_{U_{\pm_j}^2}\|\psi_4\|_{V_{\pm_4}^2}.
$$
\paragraph{(iii) The case: $\lam_2 \ll \lam_3 \sim \lam_4$} Since we have \eqref{j3-esti} with $\pm_4 =-$,  we  deal only with the case $\pm_4 = + $. By Lemmas \ref{null-esti} and \ref{bilinear} we get
\begin{align*}
J_3^h &\les \sum_{\mu \ge 1} \mu^{-2+\gam}  \left( \frac{\lam_2}{\lam_1}\right)^{\varepsilon} \mu \lam_3^{-\frac12}   \prod_{j=1}^3\|\psi_j\|_{U_{\pm_j}^2}\|\psi_4\|_{U_{\pm_4}^2} \les  \lam_3^{-\frac14} \sum_{\mu \ge 1} \mu^{-1+\gam}\lam_2^{-\de +\varepsilon} \lam_2^{\de } \mu^{-\frac14}     \prod_{j=1}^3\|\psi_j\|_{U_{\pm_j}^2}\|\psi_4\|_{U_{\pm_4}^2} \\
&\les \lam_2^{\de} \lam_3 ^{-\frac14}    \sum_{\mu \ge 1} \mu^{-\frac54+\gam - \de + \varepsilon}    \prod_{j=1}^3\|\psi_j\|_{U_{\pm_j}^2}\|\psi_4\|_{U_{\pm_4}^2} \les  \lam_2^{\de} \lam_3 ^{-\frac14} \prod_{j=1}^3\|\psi_j\|_{U_{\pm_j}^2}\|\psi_4\|_{U_{\pm_4}^2}
\end{align*}
for $0< \varepsilon < \min\{\de , \frac14\}$. In the same way as in \eqref{j3-u4} we also have
\begin{align*}
J_3^h &\les \lam_3  \prod_{j=1}^3\|\psi_j\|_{U_{\pm_j}^2}\|\psi_4\|_{U_{\pm_4}^4}.
\end{align*}
The logarithm interpolation(Lemma \ref{log}) yields that
$$
J_3^h \les (\lammed)^{\de}\prod_{j=1}^3\|\psi_j\|_{U_{\pm_j}^2}\|\psi_4\|_{V_{\pm_4}^2}.
$$

\subsubsection{Estimates for $J_4^h$} We consider the subcases: (i) $ \lam_4 \les \lam_2$ and (ii) $\lam_2 \ll \lam_4$.
\medskip
\paragraph{(i) The case: $\lam_4 \les \lam_2$} By H\"older's inequality and Lemma \ref{null-esti} we obtain
\begin{align*}
J_4^h &\les \sum_{\mu \ge 1}\mu^{-2 + \gam}\left\|\pmu Q(\psi_1,\psi_2)\right\| \left\| \widetilde{\pmu}Q(\psi_3,\psi_4)\right\| \les \sum_{\mu \ge 1} \mu^{-1 + \gam} \prod_{j=1}^3\|\psi_j\|_{U_{\pm_j}^2}\|\psi_4\|_{U_{\pm_4}^2}\\
&\les   \sum_{\mu \ge 1} \mu^{-1 + \gam - \frac{\de + \gam-1}2} \lammed^{\frac{\de + \gam-1}2}   \prod_{j=1}^3\|\psi_j\|_{U_{\pm_j}^2}\|\psi_4\|_{U_{\pm_4}^2} \les  (\lammed)^{\frac{\de + \gam-1}2}     \prod_{j=1}^3\|\psi_j\|_{U_{\pm_j}^2}\|\psi_4\|_{U_{\pm_4}^2}
\end{align*}
and by Lemma \ref{bilinear}
\begin{align*}
J_4^h &\les \sum_{\mu \ge 1} \mu^{-2 + \gam} \mu^{\frac12}\lam_3^{1-\de}\lam_4^\de \prod_{j=1}^3\|\psi_j\|_{U_{\pm_j}^2}\|\psi_4\|_{U_{\pm_4}^4}  \les (\lammed)^{\frac32}     \prod_{j=1}^3\|\psi_j\|_{U_{\pm_j}^2}\|\psi_4\|_{U_{\pm_4}^4}.
\end{align*}
By logarithm interpolation(Lemma \ref{log}) we get the desired estimates.
\medskip

\paragraph{(ii) The case: $\lam_2 \ll \lam_4$} We may consider subcases: (a) $\lam_3 \ll \lam_4 \sim \mu$ and (b) $\mu \ll \lam_3 \sim \lam_4$. But since $\lam_2 \ll \lam_4$, the case (a) does not occur. Hence we have only to consider the case (b). Due to \eqref{j4-esti} it remains to prove the estimates for the case $\pm_4 =+$. By Lemma \ref{null-esti} we get
\begin{align*}
J_4^h &\les \sum_{\mu \ge 1}   \mu^{-2+\gam} \mu^{\frac32} \lam_3^{-\frac12}      \prod_{j=1}^3\|\psi_j\|_{U_{\pm_j}^2}\|\psi_4\|_{U_{\pm_4}^2} \les  (\lammed)^{\de-\varepsilon} \lam_3^{-\varepsilon}  \sum_{\mu \ge 1}   \mu^{-1+\gam -\de +2\varepsilon}       \prod_{j=1}^3\|\psi_j\|_{U_{\pm_j}^2}\|\psi_4\|_{U_{\pm_4}^2} \\
      &\les (\lammed)^{\de-\varepsilon}    \lam_3^{-\varepsilon}   \prod_{j=1}^3\|\psi_j\|_{U_{\pm_j}^2}\|\psi_4\|_{U_{\pm_4}^2}
\end{align*}
for $0< \varepsilon <\frac{\de -\gam +1}{2}$. By Lemma \ref{bilinear} we also get
\begin{align*}
J_4^h &\les \sum_{\mu \ge 1}   \mu^{-2+\gam}  (\lam_1 \lam_2 \lam_3 \lam_4)^{\frac12}     \prod_{j=1}^3\|\psi_j\|_{U_{\pm_j}^2}\|\psi_4\|_{U_{\pm_4}^4} \les \lammed \lam_3    \prod_{j=1}^3\|\psi_j\|_{U_{\pm_j}^2}\|\psi_4\|_{U_{\pm_4}^4}.
\end{align*}
Then Lemma \ref{log} gives us the desired estimates.

\section{Nonexistence of scattering}\label{sec-nonsca}

In this section we prove Theorem \ref{non-scat}. Since the proof for backward scattering is similar to the forward scattering, we only consider the forward scattering. We proceed by contradiction, assuming that $\|\psii(0)\|_{L_x^2} > 0$.

Let us define functionals $H(t), H_\varepsilon (t) (\varepsilon > 0)$ by
	\begin{align*}
	H(t) = {\rm sgn}(c) {\rm Im} \left< \psi, \psii \right>_{L_x^2},\;\;H_\varepsilon(t) = {\rm sgn}(c) {\rm Im} \left< (1-\varepsilon\Delta)^{-1}\psi,  (1-\varepsilon\Delta)^{-1}\psii \right>_{L_x^2}.
	\end{align*}
It is clear that $H(t)$ is uniformly bounded from the mass conservation and that for each $t > 0$, $\lim_{\varepsilon \to 0}H_\varepsilon(t) = H(t)$.
Now by taking the time derivative and using the self-adjointness of Dirac operator we have
	\begin{align*}
&	\frac{d}{dt}H_\varepsilon(t)\\
 &= {\rm sgn}(c) {\rm Im} \left( \left< \partial_t(1-\varepsilon\Delta)^{-1}\psi, (1-\varepsilon\Delta)^{-1}\psii \right>_{L_x^2} +  \left<(1-\varepsilon\Delta)^{-1} \psi, \partial_t(1-\varepsilon\Delta)^{-1}\psii \right>_{L_x^2} \right)\\
	&= {\rm sgn} (c) {\rm Im} \left(  \left< -i(\al\cdot D + \beta)\psi + i V*\left<\psi, \beta \psi\right> \beta \psi, (1-\varepsilon\Delta)^{-2}\psii \right>_{L_x^2}   + \left< \psi, -i(\al\cdot D + \beta)(1-\varepsilon\Delta)^{-2}\psii \right>_{L_x^2}    \right)\\
	&= {\rm sgn}(c) {\rm Im} \left(  i \left<  V * \left<\psi, \beta \psi\right> \beta \psi, (1-\varepsilon\Delta)^{-2} \psii \right>_{L_x^2}       \right)\\
	&= |c| {\rm Re} \left<  |\cdot |^{-\gamma}*\left<\psi, \beta \psi\right> \beta \psi, (1-\varepsilon\Delta)^{-2} \psii \right>_{L_x^2}.
		\end{align*}
Integrating over $[t_*, t^*]$, we get
$$
H_\varepsilon(t^*) - H_\varepsilon(t_*) = |c| {\rm Re} \int_{t_*}^{t^*} \left<  |\cdot |^{-\gamma}*\left<\psi, \beta \psi\right> \beta \psi, (1-\varepsilon\Delta)^{-2} \psii \right>_{L_x^2}\,dt,
$$
If $\psi \in C(\mathbb R; H^\frac\gamma2)$, then Hardy-Sobolev inequality yields $|\cdot|^{-\gamma}*\left<\psi, \beta \psi\right> \beta \psi \in C([t_*, t^*]; L_x^2)$ and thus by taking the limit as $\varepsilon \to 0$ we finally get
$$
H(t^*) - H(t_*) =  \int_{t_*}^{t^*}Y(t)\,dt,
$$
where $Y(t) = |c| {\rm Re} \left<  |\cdot |^{-\gamma}*\left<\psi, \beta \psi\right> \beta \psi,  \psii \right>_{L_x^2}$.

We will show that
\begin{align}\label{lowerbound}
|Y(t)| \gtrsim t^{-\gamma}
\end{align}
for sufficiently large $t$ when $\|\psii\|_{L_x^2} > 0$.
Once \eqref{lowerbound} has been shown, since the sign of $Y$ is fixed for large time and $0< \gamma \le 1$, \eqref{lowerbound} leads us to a contradiction to the fact that $H(t)$ is uniformly bounded on time and hence completes the proof of Theorem \ref{non-scat}.

From now on we focus on the proof of \eqref{lowerbound}. To do so we reconstitute $Y$ as follows:
\begin{align*}
Y&= Y_1 + Y_2 + Y_3,\\	
Y_1 &= |c|  \mathcal V(\psii, \beta\psii),\\
	Y_2 &= |c|  {\rm Re} \left<   |\cdot|^{-\gamma}*\left(\left<\psi, \beta \psi\right> - \left<\psii, \beta \psii\right>\right) \beta \psii, \psii \right>_{L_x^2},\\
	Y_3 &= |c|  {\rm Re} \left<   |\cdot|^{-\gamma}*\left<\psi, \beta \psi\right> \beta (\psi -\psii), \psii \right>_{L_x^2}.
	\end{align*}

	We first deal with the $Y_1$.  By the assumption \eqref{assu-scat} we have that for any $t > t_*$
	\begin{align}
	\begin{aligned}\label{y1-esti}
	|Y_1| &= |c| |\mathcal V(\psii, \beta\psii)(t)| \ge \theta |c|\mathcal V(\psii, \psii)\\
	&= \theta|c| \int\Big( |\cdot|^{-\gamma}* |\psii|^2\Big)(t, x)|\psii(t, x)|^2dx\\
	&\ge \theta|c|  (4At)^{-\gam}   \left(\int   \rho\left(\frac{x}{At}\right) \left|\psii(t, x)\right|^2 dx\right)^2.
\end{aligned}
	\end{align}
	Here $\rho$ is the same cut-off function defined in the Section \ref{low-esti}. Let us set $\psii(0) = \vps$ and for the simplicity sake let us denote $\Pi_{\pm}(D)\vps$ by $\vp_\pm$. Then  $\psii = \epd \vp_+ - \emd \vp_-$ and we obtain
	\begin{align*}
\int \rho\left(\frac{x}{At}\right)   \left|\psii(x)\right|^2 dx&= \int \rho\left(\frac{x}{At}\right) |\epd \vp_+ - \emd \vp_- |^2dx\\
	&= \int   \rho\left(\frac{x}{At}\right)  \left(|\epd \vp_+|^2  +    |\emd \vp_- |^2  -2  {\rm Re} \left<\epd \vp_+, \emd \vp_- \right> \right) dx.
	\end{align*}
	We handle the last term in the integrand as follows:
	\begin{align*}
	 \int \rho\left(\frac{x}{At} \right) \left<\epd \vp_+, \emd \vp_- \right>  dx &=  \left< \rho\left(\frac{x}{At} \right)\epd \vp_+, \emd\Pi_{-}(D) \vps\right>_{L_x^2}\\
	&=  \left<\Pi_{-}(D) \left(\rho\left(\frac{x}{At} \right)  \epd \vp_+ \right), \emd \vps\right>_{L_x^2}\\
	&=  \left<\rho\left(\frac{x}{At} \right)\epd \Pi_{-}(D) \Pi_+(D)\vps,  \emd \vps\right>_{L_x^2}\\
	&\qquad\;\; + \left<  \left[\Pi_{-}(D), \rho\left(\frac{x}{At} \right)\right]  \epd \vp_+, \emd \vps\right>_{L_x^2}\\
	&= - \frac12  \left<\left[\frac{\al\cdot D + \beta}{\left< D\right>},\rho\left(\frac{x}{At} \right)\right]\epd \vp_+,   \emd \vps\right>_{L_x^2}.
	\end{align*}
	Here $\left[A ,B \right] := AB - BA$. We used \eqref{proj-commu} for the last integral. Then by Plancherel's theorem we see that
	\begin{align*}
	\left|2{\rm Re}\int \left<\epd \vp_+ , \emd\vp_-\right> dx\right| &= {\rm Re} \left|\int \left<\left[\frac{\al\cdot D + \beta}{\left< D\right>},\rho\left(\frac{x}{At} \right)\right] \vp_+, \emd \vps\right> dx\right|\\
	&\les (At)^2 \iint \frac{|\xi-\eta|}{\left<\xi\right>} |\widehat{\rho}\left(At(\xi-\eta)\right)|  |\Pi_+(\eta)\widehat{\vps}(\eta)|    |\widehat{\vps}(-\xi)| d\eta d\xi\\
	&\les (At)^{-1} \|\vps\|_{L_x^2}^2.
	\end{align*}
	From this we can choose sufficiently large $t$ so that
	\begin{align}\label{era}
	\left|2{\rm Re}\int \left<\epd \vp_+ , \emd\vp_-\right> dx\right| \le \frac1{10}\|\vps\|_{L_x^2}.
	\end{align}
	
	Let us set $\vpsl := P_{\ka^{-1}<\cdot\le \ka}\left(\rho\left(\frac{\cdot}{r}\right) \vps\right)$, $\phi(\xi) := x\cdot\xi + t\left<\xi\right>$, and $L(\xi) := \nabla_{\xi} \left(\frac{\nabla_{\xi}\phi}{|\nabla_{\xi}\phi|^2}\right)$. If $|x| \ge At$, by using integration by parts twice we have
	\begin{align}
	\begin{aligned}\label{decay-space}
	&\left|\epd \Pi_{+}(D)\vpsl(x) \right|\\
	&=  C \left|\int \nabla_{\xi}\phi e^{i\phi} \frac{\nabla_{\xi}\phi}{|\nabla_{\xi}\phi|^2} \widehat{\vp_{\ka,r}^+}(\xi) d \xi \right|\\
	&= C  \left| \int  e^{i\phi} L(\xi) \widehat{\vp_{\ka,r}^+}(\xi) d \xi  + \int  e^{i\phi} \frac{\nabla_{\xi}\phi}{|\nabla_{\xi}\phi|^2} \nabla_{\xi} \widehat{\vp_{\ka,r}^+}(\xi) d \xi\right|\\
	&= C \left| \int  e^{i\phi} L(\xi)^2 \widehat{\vp_{\ka,r}^+}(\xi) + e^{i\phi} \nabla_{\xi} L(\xi) \frac{\nabla_{\xi}\phi}{|\nabla_{\xi}\phi|^2}  \widehat{\vp_{\ka,r}^+}(\xi) + e^{i\phi} L(\xi)^2 \nabla_{\xi}\widehat{\vp_{\ka,r}^+}(\xi) d \xi \right| \\
	&\qquad +  C\left|\int  e^{i\phi} \nabla_{\xi} \left(\frac{1}{|\nabla_{\xi}\phi|^2}\right) \nabla_{\xi} \widehat{\vp_{\ka,r}^+}(\xi)  +   e^{i\phi}  \frac{1}{|\nabla_{\xi}\phi|^2} \nabla_{\xi}^2 \widehat{\vp_{\ka,r}^+}(\xi) d \xi\right|\\
	&\le C(A,\ka,r)|x|^{-2}\|\vp_+\|_{L_x^2}.
\end{aligned}
	\end{align}
 Since for large $\ka$ and $r$, $\|\vps - \vpsl\| \le \frac1{10}\|\vps\|$, by \eqref{era} and \eqref{decay-space}, we have
	\begin{align*}
	\left\|\rho\left(\frac{\cdot}{At}\right) \psii\right\|_{L_x^2}^2 &\ge \left\|\rho\left(\frac{\cdot}{At}\right)  \epd \Pi_{+}(D)\vps \right\|_{L_x^2}^2 - \frac{1}{10}\|\vps\|_{L_x^2}^2\\
	&\ge \left\|\rho\left(\frac{\cdot}{At}\right) \epd \Pi_{+}(D)\vpsl \right\|_{L_x^2}^2 - \frac{11}{50}\|\vps\|_{L_x^2}^2 - \frac{1}{10}\|\vps\|_{L_x^2}^2\\
	&\ge \|\epd \Pi_{+}(D) \vpsl\|_{L_x^2}^2 - \int \left( 1- \rho\left(\frac{\cdot}{At}\right)  \right) |\epd\Pi_{+}(D) \vpsl|^2 dx - \frac{8}{25}\|\vps\|_{L_x^2}^2\\
	&\ge \frac{17}{25}\|\vps\|_{L_x^2}^2 - \frac9{10}C(M,r)\|\vps\|_{L_x^2}^2\int_{|x|\ge At} |x|^{-4}dx\\
	&\ge \frac35\|\vps\|_{L_x^2}^2.
	\end{align*}
	This together with \eqref{y1-esti} leads us to
	\begin{align}\label{y1}
	|Y_1|  \gtrsim |t|^{-\gam} .
	\end{align}

	Now let us turn to $Y_2,\;Y_3$. For this purpose we introduce time decay estimate for the linear solutions and $L_x^\infty$ estimates for the potential term.

	\begin{lem}[see Lemma 4.2 of \cite{choz}]\label{time-decay}
		Let $f \in B_{1,1}^2$. Then
		\begin{align*}
		\|\epm f\|_{L_x^\infty} \les t^{-1} \|f \|_{B_{1,1}^2}.
		\end{align*}
	\end{lem}
Here $B_{1,1}^2$ is the inhomogeneous Besov space defined by $\left\{ f : \|f\|_{B_{1,1}^2}:= \sum_{\lam\ge 1}\lam^2 \|P_{\lam} f\|_{L_x^1}  < \infty\right\}$.

\begin{lem}\label{infty-esti}
	Let $0< \gam < 1$. Then for any $\mathbb C$-valued functions $u_1 \in L_x^2,\,u_2 \in L_x^2 \cap L_x^\infty$
	\begin{align}
	\||x|^{-\gam} * (u_1 \overline{u_2}) \|_{L_x^\infty} \les \|u_1\|_{L_x^2} \|u_2\|_{L_x^2}^{1-\gam} \|u_2\|_{L_x^\infty}^{\gam}.
	\end{align}
\end{lem}
\begin{proof}
	Using H\"older's inequality we get for any  $r > 2$
	\begin{align*}
	\||x|^{-\gamma} * (u_1 \overline{u_2})\|_{L_x^\infty} &=  \sup_{x\in \mathbb{R}^2}\left|\int |x-y|^{-\gam} |u_1(y) \overline{u_2}(y)|  dy\right| \\
	&\les \||x|^{-\gam}\|_{L_x^2(|x|\le R)}\|u_1\|_{L_x^2} \|u_2\|_{L_x^\infty} + \||x|^{-\gam}\|_{L_x^\frac{r}\gamma(|x|\ge R)}\|u_1\|_{L_x^2} \|u_2\|_{L_x^\frac{2r}{r-2\gamma}}\\
	&\les R^{1- \gam}\|u_1\|_{L_x^2} \|u_2\|_{L_x^\infty} + R^{\frac{\gamma(2 - r)}{r}}\|u_1\|_{L_x^2}\|u_2\|_{L_x^2}^{1-\frac{2\gamma}{r}} \|u_2\|_{L_x^\infty}^{\frac{2\gamma}{r}}.	
	\end{align*}
	By taking $R = \|u_2\|_{L_x^2} \|u_2\|_{L_x^\infty}^{-1}$ we get the desired result.
\end{proof}

If $u_1 = u_2$, then  we can cover the case $\gamma = 1$ as follows.
\begin{lem}\label{infty-esti-1}
For any $\mathbb C$-valued functions $u \in L_x^2 \cap L_x^\infty$ we have
	\begin{align}
	\||x|^{-1} *|u|^2 \|_{L_x^\infty} \les \|u\|_{L_x^2} \|u\|_{L_x^\infty}.
	\end{align}
\end{lem}
\begin{proof}
	For any $1 < r_1 < 2$ and $r_2 > 2$ we have
	\begin{align*}
	\||x|^{-1} * |u|^2\|_{L_x^\infty} &=  \sup_{x\in \mathbb{R}^2}\left|\int |x-y|^{-\gam} |u(y)|^2  dy\right| \\
	&\les \||x|^{-1}\|_{L_x^{r_1}(|x|\le R)}\|u\|_{L_x^{r_1'}} \|u\|_{L_x^\infty} + \||x|^{-1}\|_{L_x^{r_2}(|x|\ge R)}\|u\|_{L_x^2} \|u\|_{L_x^\frac{2r_2}{r_2-2}}\\
	&\les R^{\frac2{r_1}- 1}\|u\|_{L_x^2}^{2- \frac2{r_1}}\|u\|_{L_x^\infty}^\frac2{r_2} + R^{\frac2{r_2} - 1}\|u\|_{L_x^2}^{2-\frac{2\gamma}{r_2}} \|u\|_{L_x^\infty}^{\frac{2}{r_2}}.	
	\end{align*}
	By taking $R = \|u_2\|_{L_x^2} \|u_2\|_{L_x^\infty}^{-1}$ we finish the proof.
\end{proof}

If $u_1 \neq u_2$, then we need a regularity and a space decay assumption. Then we have
\begin{lem}\label{infty-weight}
	Let $0 < s, \delta  \ll 1$. Then for any $\mathbb C$-valued functions $u_1 \in H^{s, \delta}$ and $u_2 \in L_x^\infty$
	\begin{align}
	\||x|^{-1} * (u_1 \overline{u_2}) \|_{L_x^\infty} \les \|u_1\|_{H^{s, \delta}} \|u_2\|_{L_x^\infty}.
	\end{align}
\end{lem}
\begin{proof}
By H\"older's inequality and Sobolev embedding we see that
	\begin{align*}
	\||x|^{-1} * (u_1 \overline{u_2})\|_{L_x^\infty} &\les \||x|^{-1}\|_{L_x^\frac2{1+s}(|x|\le 1)}\|u_1\|_{L_x^\frac{2}{1-s}} \|u_2\|_{L_x^\infty} + \||x|^{-1}\|_{L_x^\frac{2}{2-\frac\delta2}(|x|\ge 1)}\|u_1\|_{L_x^\frac2{1+\frac\delta2}(|x|\ge 1)} \|u_2\|_{L_x^\infty}\\
	&\les \|u_1\|_{H^s} \|u_2\|_{L_x^\infty} + \|(1+|x|)^\delta u_1\|_{L_x^2} \|u_2\|_{L_x^\infty} \les \|u_1\|_{H^{s, \delta}}\|u_2\|_{L_x^\infty}.	
	\end{align*}
	\end{proof}

 If $0 < \gamma < 1$, then from Lemmas \ref{infty-esti}, \ref{time-decay} and mass conservation it follows that
	\begin{align*}
	|Y_2| &\les   \left| \left<   V*\Big(\left<\psi, \beta \psi\right> - \left<\psii, \beta \psii\right>\Big) \beta \psii, \psii \right>_{L_x^2}\right|\\
	&\les  \left| \int  \Big(\left<\psi, \beta \psi\right> - \left<\psii, \beta \psii\right>\Big) V*\left<\beta \psii, \psii \right> dx\right|\\
	&\les  \|\psi - \psii\|_{L_x^2} \left( \|\psi\|_{L_x^2} + \|\psii\|_{L_x^2} \right)\|\psii\|_{L_x^2}\|\psii\|_{L_x^2}^{1-\gam}\|\psii\|_{L_x^\infty}^{\gam}\\
	&\les \|\psi - \psii\|_{L_x^2} \left( \|\psi_0\|_{L_x^2} + \|\vps\|_{L_x^2} \right)\|\vps\|_{L_x^2}^{2-\gam}t^{-\gam}\left(\|\vp_+\|_{B_{1,1}^2} + \|\vp_-\|_{B_{1,1}^2}\right)^{\gam}
	\end{align*}
	and
	\begin{align*}
	|Y_3| &\les \left| \left<   V*\left<\psi, \beta \psi\right> \beta (\psi -\psii), \psii \right>_{L_x^2} \right|\\
	&\les \left| \int   \left<\psi, \beta \psi\right> \left(V* \left<\beta (\psi -\psii),\psii \right>\right)  dx\right|\\
	&\les \|\psi\|_{L_x^2}^2 \|\psi-\psii\|_{L_x^2}\|\psii\|_{L_x^2}^{1-\gam}\|\psii\|_{L_x^\infty}^\gam\\
	&\les \|\psi_0\|_{L_x^2}^{2}\|\psi-\psii\|_{L_x^2}\|\vps\|_{L_x^2}^{1-\gam}t^{-\gam}\left(\|\vp_+\|_{B_{1,1}^2} + \|\vp_-\|_{B_{1,1}^2} \right)^{\gam}.
	\end{align*}
	Therefore we get
	\begin{align}\label{yj}
	|Y_j| = o(t^{-\gam})
	\end{align}
	for $j=1,2$. Then \eqref{y1} and \eqref{yj} conclude \eqref{lowerbound} for the case $0 < \gamma < 1$.
	
If $\gamma = 1$, then by Lemma \ref{infty-esti-1} we have
\begin{align*}
	|Y_2| &\les  \left| \int  \Big(\left<\psi, \beta \psi\right> - \left<\psii, \beta \psii\right>\Big) |\cdot|^{-1}*\left<\beta \psii, \psii \right> dx\right|\\
	&\les  \|\psi - \psii\|_{L_x^2} \left( \|\psi\|_{L_x^2} + \|\psii\|_{L_x^2} \right)\|\psii\|_{L_x^2}\|\psii\|_{L_x^\infty}\\
	&\les \|\psi - \psii\|_{L_x^2} \left( \|\psi_0\|_{L_x^2} + \|\vps\|_{L_x^2} \right)\|\vps\|_{L_x^2}t^{-1}\left(\|\vp_+\|_{B_{1,1}^2} + \|\vp_-\|_{B_{1,1}^2}\right) = o(t^{-1})
	\end{align*}	
and by Lemma \ref{infty-weight} and mass conservation
\begin{align*}
	|Y_3| &\les \left| \left<   |\cdot|^{-1}*\left<\psi, \beta \psi\right> \beta (\psi -\psii), \psii \right>_{L_x^2} \right|\\
	&\les \left| \int   \left<\psi, \beta \psi\right> \left(V* \left<\beta (\psi -\psii),\psii \right>\right)  dx\right|\\
	&\les \|\psi\|_{L_x^2}^2 \|\psi-\psii\|_{H^{s,\delta}}\|\psii\|_{L_x^\infty}\\
	&\les \|\psi_0\|_{L_x^2}^{2}\|\psi-\psii\|_{H^{s, \delta}}t^{-1}\left(\|\vp_+\|_{B_{1,1}^2} + \|\vp_-\|_{B_{1,1}^2} \right) = o(t^{-1}).
	\end{align*}
This completes the proof of Theorem \ref{non-scat}.

\section*{Acknowledgements}
This work was supported in part by NRF-2018R1D1A3B07047782(Republic of Korea).


\end{document}